\documentclass{amsart}
\usepackage{comment}
\usepackage{euscript,graphicx,epstopdf,amscd,amsgen,amsfonts,amssymb,latexsym,amsmath,amsthm,graphicx,mathrsfs,times,color,overpic,bbm}

\usepackage{xcolor}
\definecolor{forestgreen}{rgb}{0.0, 0.27, 0.13}

\usepackage{chngcntr}%
\counterwithin{equation}{section}%

\usepackage{mathrsfs}  
\usepackage{psfrag}
\usepackage{mathtools}
\usepackage{color}
\usepackage{todonotes}
\usepackage{enumitem}\usepackage{soul}
\usepackage{hyperref}
\hypersetup{colorlinks=true, citecolor=blue, linkcolor=blue, hypertexnames=false}

\theoremstyle{plain}%
\newtheorem{main}{Theorem}%

\newtheorem{theorem}{Theorem}[section]
\newtheorem{lemma}[theorem]{Lemma}
\newtheorem{corollary}[theorem]{Corollary}
\newtheorem{proposition}[theorem]{Proposition}

\newtheorem{definition}[theorem]{Definition}
\newtheorem{question}{Question}

\newtheorem{remark}[theorem]{Remark}

\def\SP{\mathcal{SP}^r_3(M)}
\def\DA{\mathcal{DA}^r_2(\mathbb{T}^4)}

\def\bN{\mathbb{N}}
\def\bZ{\mathbb{Z}}

\def \cF {{\mathcal F}}

\def \cM {{\mathcal M}}

\def \cO {{\mathcal O}}

\def \cR {{\mathcal R}}

\def \cU {{\mathcal U}}
\def \cV {{\mathcal V}}

\def\Diff{{\rm Diff}}
\def\dim{{\rm dim}}

\def\Per{{\rm Per}}

\DeclareMathOperator{\dime}{dim}
\numberwithin{equation}{section}

\DeclareMathSymbol{\varnothing}{\mathord}{AMSb}{"3F}
\title[PH with a finite number of MME]{Partially hyperbolic diffeomorphisims with a finite number of measures of maximal entropy}

\author[J. Mongez]{Juan Carlos Mongez}
\address{Departamento de Matem\'atica, Universidade Federal do Cear\'a (UFC), Campus do Pici,
Bloco 914, CEP 60455-760. Fortaleza -- CE, Brasil}
\email{juanmongez1994@gmail.com}

\author[M. Pacifico]{Maria Jose Pacifico}
\address{Instituto de Matem\'atica, Universidade Federal do Rio de Janeiro, Cidade Universit\'aria - Ilha do Fund\~ao, Rio de Janeiro 21945-909,  Brazil}
\email{pacifico@im.ufrj.br}

\author[M. Poletti]{Mauricio Poletti}
\address{Departamento de Matem\'atica, Universidade Federal do Cear\'a (UFC), Campus do Pici,
Bloco 914, CEP 60455-760. Fortaleza -- CE, Brasil}
\email{mpoletti@mat.ufc.br}

\begin{document}

\begin{abstract}
We prove the finiteness of ergodic measures of maximal entropy for partially hyperbolic diffeomorphisms where the center direction has a dominated decomposition into one dimensional bundle and there is a uniform lower bound for the absolute value of the Lyapunov exponents. As applications we prove finiteness for a class derived from Anosov partially hyperbolic diffeomorphisms defined on $\mathbb{T}^4$ and that in a class of skew product over partially hyperbolic diffeomorphisms there exists a $C^1$ open and $C^r$ dense set of diffeomorphisms with a finite number of ergodic measures of maximal entropy. We also study the upper semicontinuity of the number of measures of maximal entropy with respect to the diffeomorphism.
\end{abstract}

\thanks{The authors were partially supported by CAPES-Finance Code 001. 
JCM was partially supported by Instituto Serrapilheira, grant ``Jangada Din\^{a}mica: Impulsionando Sistemas Din\^{a}micos na Regi\~{a}o Nordeste and CNPq-Brazil-Bolsa de Pós-doutorado Júnior No. 175065/2023-3.
MJP was partially supported by FAPERJ Grant CNE No. E-26/202.850/2018(239069), Grant CNPq-Brazil No. 307776/2019-0 and CNPq-Projeto Universal No. 404943/2023-3.
MP was supported by Instituto Serrapilheira, grant number Serra-R-2211-41879 and FUNCAP, grant AJC 06/2022.}

\maketitle
\tableofcontents

\section{Introduction}


Sinai, Ruelle, and Bowen pioneered the study of measures of maximal entropy in the 1970s. Sinai made a remarkable advance by proving the existence and uniqueness of measures of maximal entropy for Anosov diffeomorphisms \cite{sin72}. 
Later, Ruelle and Bowen expanded this framework to include uniformly hyperbolic systems (Axiom A), further advancing the theory \cite{Bow71, Rue68, Rue78}.

Since then, the study of the existence and finiteness of measures of maximal entropy has remained an active area of research. In a celebrated work, using Yomdin theory~\cite{Yom87}, Newhouse proved that $C^\infty$ diffeomorphisms on a compact manifold always admit measures of maximal entropy \cite{New89}, for finite regularity this was extended by Burguet~\cite{burguet2012symbolic} for diffeomorphisms with large topological entropy. Recently, Buzzi, Crovisier, and Sarig established that, on surfaces, $C^\infty$ diffeomorphisms with positive topological entropy have a finite number of measures of maximal entropy \cite{BCS22}

The last result does not extend to higher dimensions. For instance, consider the product of an Anosov diffeomorphism and the identity; such examples fail to replicate these properties. Additionally,  measures of maximal entropy cannot be guaranteed for systems with lower differentiability \cite{Mis73,B14}.

Once it is established that uniform hyperbolic diffeomorphisms have a finite number of measures of maximal entropy, it is natural to investigate if partially hyperbolic diffeomorphisms (see Definition \ref{Def-PHD}) have the same property.

In this paper, we study the number of ergodic measures of maximal entropy of partially hyperbolic diffeomorphisms that exhibit a dominated splitting (see definition \ref{def-DS}).

Let $f: M \to M$ be a continuous map on a compact metric space $M$. Denoting by $\mathcal{M}_f$ the set of $f$-invariant probability measures, the variational principle establishes the following relationship:  
\begin{equation}\label{e-variacional-entropia}
h(f) = \sup\{h_\mu(f) : \mu \in \mathcal{M}_f\},
\end{equation}
where $h(f)$ is the topological entropy of $f$, and $h_\mu(f)$ is the metric entropy of $f$ with respect to the invariant measure $\mu$.

\begin{definition}\label{mme}
An ergodic measure of maximal entropy (MME for short) is an ergodic measure
 $\mu \in \cM_f$  that $h_\mu(f)$ achieves the supremum in equation (\ref{e-variacional-entropia}).
\end{definition}

\subsection*{Dominated Splittings and Partially Hyperbolic Diffeomorphisms}  We consider $M$ as a compact Riemannian manifold, and $f: M \to M$ as a $C^1$ diffeomorphism.

\begin{definition}\label{def-DS}
We say that $f$ has a dominated splitting $E^{cs} \oplus E^{cu}$ when, for some appropriated Riemannian metric, it holds:

\begin{enumerate}
    \item[(i)] the bundles $E^{cs}$ and $E^{cu}$ are both $Df$-invariant,
    \item[(ii)] the bundles $E^{cs}(x)$ and $E^{cu}(x)$ have dimensions independent of $x \in M$,
    \item[(iii)] there exist $C > 0$ and $0 < \lambda < 1$ such that:
    $$ \| Df^n|_{E^{cs}(x)} \| \cdot \| Df^{-n}|_{E^{cu}(f^n(x))} \| \leq \lambda^n, $$
    for all $x \in M$ and $n \geq 0$.
\end{enumerate}

\end{definition}

In general, a $Df$-invariant splitting of the form:  $T M = E_1 \oplus \cdots \oplus E_k $
is dominated if for all $i = 1, \ldots, k-1$, the splitting $TM = E_1^i \oplus E_{i+1}^k$ is dominated, where $E^\ell_j = E_j \oplus \cdots \oplus E_\ell$ for $1 \leq j \leq \ell \leq k$.

\begin{definition}\label{Def-PHD}
    A diffeomorphism $f$ is  \textit{partially hyperbolic} if it admits a dominated splitting $TM = E^{ss} \oplus E^c \oplus E^{uu}$, where $E^{ss}$ is uniformly contracted by $Df$ and $E^{uu}$ is uniformly contracted by $Df^{-1}$.
\end{definition}

 The first two authors identified an open set of partially hyperbolic diffeomorphisms with a one-dimensional center bundle that admits a finite number of measures of maximal entropy \cite{MP24}. In particular, they demonstrated that if all measures of maximal entropy for these diffeomorphisms have central Lyapunov exponents uniformly bounded away from zero, that is, are $\alpha$--hyperbolic with $\alpha > 0$ (see \ref{Hyperbolicmeasure} for the definition of an $\alpha$--hyperbolic measure), then the number of such measures is finite. 

For partially hyperbolic diffeomorphisms, it is well known that stable and unstable bundles, $E^{ss}$ and $E^{uu}$, uniquely integrate into strong stable foliation $\mathcal{F}^{ss}$ and strong unstable foliation $\mathcal{F}^{uu}$, respectively. 
For each $x \in M$, we denote the strong stable leaf containing $x$ by $W_f^{ss}(x)$, and similarly, the strong unstable leaf containing $x$ by $W_f^{uu}(x)$, when there is no confusion we omit the subindex $f$.

The topological unstable entropy (entropy along the foliation $\cF^{uu}$) was introduced in \cite{HHW17, Yan21}, see section~\ref{ss.unstable-entropy} for the definition. We always have that the unstable topological entropy ($h^u_{\mathrm{top}}(f)$) is less than or equal to the topological entropy. 
Our first result proves that for certain partially hyperbolic diffeomorphisms, the strict inequality implies finiteness measures of maximal entropy. Moreover, we have upper semi-continuity on the number of measures of maximal entropy with respect to the diffeomorphism.

\begin{main}\label{theo-entropy-condition-dimension-2}
Let $f: M \to M$ be a $C^r$, $r>1$, diffeomorphism on a compact manifold $M$ admitting a dominated splitting $E^{ss} \oplus E_1 \oplus E_2 \oplus E^{uu}$, where $E^{ss}$ is uniformly contracting, $E^{uu}$ is uniformly expanding, and each $E_i$ is one-dimensional. Suppose that $h(f) > \max\{h^u(f), h^s(f)\}$.  Then, there exists a $C^1$ neighborhood $\cU$ of $f$ such that:
\begin{enumerate}
\item[(a)] every $C^{1+}$ diffeomorphism $g \in \cU$ admits only a finite number of MMEs;
\item[(b)] the number of these measures for $g$ is bounded above by the number of MMEs for~ $f$.
\end{enumerate}

\end{main}

Here $h^s(f) $ denotes the topological stable entropy of $f$ and it is defined as $h^u(f^{-1})$. This result extends \cite{MP24} to two-dimensional center bundle with a dominated splitting.

\begin{remark}
    It can be proved that the maps in Theorem~\ref{theo-entropy-condition-dimension-2} satisfy conditions called entropy hyperbolicity and entropy continuity defined by Buzzi, Crovisier and Sarig in a recent result \cite{BCS-SPR}. 
    Their result implies that these maps are Strong Positive Recurrent, allowing to conclude part (a) of Theorem~\ref{theo-entropy-condition-dimension-2}.
Since our result was obtained independently, and our proof is simpler for the case we are dealing with, we choose to retain it.
\end{remark}

As an application we prove finiteness of MME for a class of Derived from Anosov diffeomorphisms as introduced by Ma\~n\'e in \cite{M78}.

\subsection*{Partially hyperbolic diffeomorphisms derived from Anosov.}
Let $f:\mathbb{T}^n\to \mathbb{T}^n$ be a partially hyperbolic diffeomorphism, it is said to be \emph{derived from Anosov} if its linear part $A\in \operatorname{SL}(n,\mathbb{Z})$ is hyperbolic. 

There are many results of the uniqueness of MME for derived from Anosov diffeomorphisms with one-dimensional center, i.e, with a decomoposition $E^{ss}\oplus E^c\oplus E^{uu}$, $\operatorname{dim} E^c=1$, see for example \cite{CR-TA}, \cite{FPS14}, \cite{MP23},\cite{ures2012intrinsic}.

Let $\DA$, $r>1$, be the set of partially hyperbolic $C^r$ diffeomorphisms derived from Anosov $f:\mathbb{T}^4\mapsto \mathbb{T}^4$ with dominated splitting $E^{ss}\oplus E_1\oplus E_2 \oplus E^{uu}$, where each bundle is one-dimensional, and its linear part $A$ has eigenvalues satisfying $|\lambda^{ss}|<|\lambda^{s}|<1<|\lambda^{u}|<|\lambda^{uu}|$.

 Buzzi et al \cite{BUZZI_FISHER_SAMBARINO_VÁSQUEZ_2012}, found open sets of Derived from Anosov diffeomorphisms having a unique MME, this includes an open subset of $\DA$.  Carrasco et al \cite{LIZ-PUJ} proved that for $A$ as above, if there are  some extra conditions on the integrability of the invariant bundles, then $h(f)=h(A)$. 
 Also, Alvarez et al \cite{AL-SA-VA} proved that if $f\in \DA$ satisfies that the pre-image of a point by the semi conjugation with the linear part has a good geometric structure and the set where the conjugation fails to be injective has zero measure then $f$ has a unique MME.

For the following result, we do not impose any assumptions on the semi-conjugacy to the linear part. We only require that the strong stable and unstable foliations of  $f\in\DA$ lift to \emph{quasi-isometric} foliations (see definition \ref{def-quasi-isometric}). By \cite{FPS14} this is satisfied for every $f\in \DA$ isotopic to its linear part by a path of partially hyperbolic diffeomorphisms.

\begin{main}\label{theo-DA}
    Let $f\in \DA$ be such that the lift of the stable and unstable foliations to the universal cover are quasi-isometric. Then $f$ has a finite number of MMEs. 
    Moreover, this number is an upper bound on the number of MMEs in a $C^1$ neighborhood of $f$.
    
This is true, in particular, for $f\in \DA$  isotopic to $A$ along a path of partially hyperbolic diffeomorphisms with one-dimensional strong stable and strong unstable bundles.

\end{main}
For derived from Anosov diffeomorphisms with decomposition $E^{ss}\oplus E^c\oplus E^{uu}$ it was proved by \cite{FPS14} that if the diffeomorphism is isotopic to its linear part by a path of partially hyperbolic diffeomorphisms with the same decomposition then it has a unique MME.

\begin{question}
    In the set of diffeomorphisms isotopic to $A$ by a path of partially hyperbolic with the same splitting $E^{ss}\oplus E_1\oplus E_2\oplus E^{uu}$ does every diffeomorphism have a unique MME?
    If not, which are the diffeomorphisms where the bifurcation happens? 
\end{question}

In Theorem \ref{theo-entropy-condition-dimension-2}  all measures of high entropy have the same index (see Definition \ref{def-index}).
This requirement is also part of the entropy hyperbolicity condition of \cite{BCS-SPR}. However, we can also address certain cases where measures of maximal entropy admit different indices:

\begin{main}\label{MainTheo2index}
Let $f: M \to M$ be a $C^{1+}$ diffeomorphism on a compact manifold $M$ admitting a dominated splitting $E^{ss} \oplus E_1 \oplus \cdots \oplus E_k \oplus E^{uu}$, where $E^{ss}$ is uniformly contracting, $E^{uu}$ is uniformly expanding, and each $E_i$ is one-dimensional. Suppose there exist constants $\alpha > 0$ and $j \in \{1, \dots, k\}$ such that the following conditions hold:
\begin{enumerate}
    \item All MME are $\alpha$--hyperbolic.
\item There are no MME with an index greater than $j+dim(E^{ss})$.
\end{enumerate}
Then $f$ admits at most a finite number of MMEs of index $j+dim(E^{ss})$.
\end{main}

In particular, when the central direction has dimension 3 and $\lambda_2(\mu)$ denotes the Lyapunov exponent in the direction of $E_2$ with respect to $\mu$, we have the following result:

\begin{main}\label{theo-entropy-condition-different index}
Let $f: M \to M$ be a $C^r$, $r>1$, diffeomorphism on a compact manifold $M$ admitting a dominated splitting $E^{ss} \oplus E_1 \oplus  E_2 \oplus E_3\oplus E^{uu}$, where $E^{ss}$ is uniformly contracting, $E^{uu}$ is uniformly expanding, and each $E_i$ is one-dimensional.
Such that:
\begin{itemize}
    \item $h(f)>\max\{h^u(f),h^s(f)\}$.
    \item There exists $\alpha>0$ such that every MME has $|\lambda_2(\mu)|\geq \alpha$.
\end{itemize} 
Then $f$ has a finite number $k$ of MME. 

Moreover, for every $\chi>0$, there exists a $C^1$ neighborhood of $f$, $\mathcal{V}_\chi$ such that for every $g\in C^r\cap \mathcal{V}_\chi$, the number of $\chi$-hyperbolic MMEs of $g$ is bounded by $k$.
\end{main}

If we have  the same uniform bound for the absolute value of $\lambda_2(\mu)$ in a neighborhood of $f$,  by the second part  of the above theorem, we have upper semi continuity on the number of MMEs.

\subsection*{Skew product over partially hyperbolic diffeomorphisms.}
In \cite{HHTU12}, Hetz et al proved that a partially hyperbolic skew product with circle fibers over an Anosov diffeomorphism generically has a finite number of MMEs. 
These examples were  further studied in \cite{tahzibi-yang-IP} and \cite{ures2021maximal}. Similar results were proved for perturbations of the time one map of Anosov flows, see \cite{BFT} and \cite{CP}. 
Here we study the case of skew product with circle fiber over partially hyperbolic diffeomorphisms.

Let $f:M\to M$ be as in Theorem~\ref{theo-entropy-condition-dimension-2} with a dominated splitting $E^{ss}_f\oplus E^-_f\oplus E^+_f\oplus E^{uu}_f$, where  $E^+$ and $E^-$ are identified as $E_1$ and $E_2$ respectively.

Let  $F:M\times S^1\to M\times S^1$ be a skew product $F(x,t)=(f(x),f_x(t))$ such that there exists $0<\gamma<1$ satisfying

\begin{equation}\label{eq.skewP}
\|Df(x)|_{E^-_f}\|\gamma^{-1}\leq \sup_{t\in S^1 }|f'_x(t)| \leq \|Df(x)|_{E^+_f}\|\gamma.    
\end{equation}
Let $\SP$ be the set of $C^r$, $r>1$, skew products as above. Then we have the following result.

\begin{main}\label{theo-entropy-Skew-Product}
     There is a $C^1$ open and $C^r$ dense subset of diffeomorphisms $\cU \subset\SP$ having a finite number of MMEs, all of them hyperbolic. 
     
Moreover, for every $F\in \cU$ and $\chi>0$, there exists a $C^1$ neighborhood $\mathcal{V}_\chi\ni F$ such that the number of MMEs of $F$ is an upper bound for  the number of $\chi$-hyperbolic MME of $G\in \cV_\chi$.
\end{main}

To prove Theorem~\ref{theo-entropy-Skew-Product} we find an open and dense set of diffeomorphisms such that each of them satisfies the hypothesis of Theorem~\ref{theo-entropy-condition-different index}. In particular, we prove that each diffeomorphism there has a lower bound on the center Lyapunov exponent. If one can find a uniform  lower bound for the center Lyapunov exponent in the neighborhood of $F$ then, by the second part of the above theorem, the number of MMEs will be upper semi continuous on $F$.

\begin{question}
    Can we find an open and dense set where the number of MME varies upper semi-continuously? Can we have $F_k\to F$, with $F$ in our open and dense set, such that $F_k$ has a MME $\tilde{\mu}_k$ with $\lambda^c(\tilde{\mu}_k)\to 0$?
\end{question}

In \cite{HHTU12} (for $F$ a skew product over Anosov) and \cite{CP15} (for $F$ a perturbation of the time-one map of an Anosov flow), it is proven that under accessibility, if there exists a measure of maximal entropy with a zero center exponent, then it is unique. Moreover, this corresponds to a rigid case: a rotation-type behavior  in \cite{HHTU12}  and the time-one map of a flow in \cite{CP15}.

\begin{question}
    If $F\in \SP$ is accessible, for each measure of maximal entropy of $f$ can we have more than one MME with zero center exponent that projects on it? Do we have some rigidity in the case of zero exponent?
\end{question}
    
\subsection*{Idea of the proofs} 

In this paper, we use the Buzzi-Crovisier-Sarig criterion, which establishes that, for diffeomorphisms, hyperbolic MMEs cannot be homoclinically related in the sense of definition \ref{def-HR-measures}. The last author, Lima and Obata, generalized this criterion for maps that can be noninvertible and admit singularities, \cite{LOP}. 
Similar ideas were used in \cite{obata2021uniqueness}, \cite{MP24}, \cite{lima2024homoclinic}. 

For \textbf{Theorem \ref{theo-entropy-condition-dimension-2}}, under the assumption that $h(f) > \max\{h^u(f), h^s(f)\}$, we prove that every MME of $f$ has central Lyapunov exponents bounded away from zero. Moreover, showing that this condition is open, we extend these results to any $C^{r}, r> 1,$ diffeomorphism in an open neighborhood of $f$ in the $C^1$ topology. In particular, we prove that there exists $\beta > 0$ such that every MME is $\beta$--hyperbolic with index 
$\dim E^s + 1$.
    
Using the fact that metric entropy is upper semicontinuous, we prove that every accumulation point of the sequence of MMEs $\mu$ is a measure of maximal entropy. So, the Lyapunov exponents are strictly greater than $\beta$ and have the same index.
    
 We use this to construct stable Pesin blocks such that any MME sufficiently close to $\mu$ assigns a measure arbitrarily close to one to the same block.

Since points in Pesin blocks have invariant Pesin manifolds of uniform size, if there are infinite MMEs, we conclude that at least two distinct MMEs must be homoclinically related. However, this contradicts the Buzzi-Crovisier-Sarig criterion, completing the proof of finiteness.

Once  we have established that $f$ has a finite number of MMEs, we denote them by $\nu_1, \dots, \nu_k$. 
For each $\nu_i$, we take a hyperbolic periodic point $p_i(f)$  that  is homoclinically related to $\nu_i$. We consider the hyperbolic continuation $p_i(g)$ for $g$. To bound the number of MMEs of $g$ by the number of MMEs of $f$, we prove that each MME of $g$ is homoclinically related to at least one $p_i(g)$, thereby establishing an injective map from the set of MMEs of $f$ to the set of MMEs of $g$ for $g$ close enough to $f$.

For the proof of \textbf{Theorem \ref{theo-DA}}, we prove that when the strong unstable foliation is quasi-isometric, then $h^u(f) = h^u(A)$, where $A$ is the linear part of $f$. Then, the first part follows as a corollary of Theorem \ref{theo-entropy-condition-dimension-2}. 

For the second part, we recall that in \cite{FPS14}, it was proved that the strong stable and unstable foliations of each diffeomorphism $f \in \DA$, which is isotopic to $A$ along a path of partially hyperbolic diffeomorphisms, are quasi-isometric.
 
The main challenge in \textbf{Theorem \ref{MainTheo2index}} compared to \textbf{Theorem \ref{theo-entropy-condition-dimension-2}} is that the accumulation points of the MMEs of $f$ may be measures of maximal entropy that are not ergodic. As a result, points with different indices may co-exist, making it difficult to identify Pesin blocks with positive measure for MMEs near the accumulation point.

Assuming there are infinite MMEs, we consider  an accumulation point $\mu$ of the MMEs of $f$. We prove that $\mu$ is a measure of maximal entropy, so its ergodic decomposition consists of MMEs. 
By hypothesis we have that $E_j(x) \oplus \cdots \oplus E^{uu}(x)\subset T_xW^u(x)$ $\mu$-almost every point $x\in M$. Thus, we can construct an unstable Pesin block for each accumulation point $\mu$ of MMMEs with the measure as close as we want to one for the measure that is close enough to $\mu$.
    
We also construct a stable Pesin block $B^s_\ell$ with positive measure, ensuring the existence of $\beta > 0$ such that every MME assigns measure at least $\beta$ to the stable Pesin block. With this we guaranty that the Pesin block, $B^u_\ell \cap B^s_\ell$, has uniform positive measure for measures close to an accumulation point of MMEs. This allows us to replicate the argument from the first part of Theorem \ref{theo-entropy-condition-dimension-2} to establish the finiteness of MMEs.

The first part of \textbf{Theorem \ref{theo-entropy-condition-different index}} follows as a consequence of Theorem \ref{MainTheo2index}. 
To prove the second part, we adopt the approach used in the second part of the Theorem \ref{theo-entropy-condition-dimension-2}.
 This case is more delicate because of the possibility of appearing points with different indices and the lack of a uniform bound for the exponents in a neighborhood of $f$.

%
%
%

To prove \textbf{Theorem \ref{theo-entropy-Skew-Product}} we use the fact that the base diffeomorphism has finite MMEs and for each MME we use \cite{Ova18} and \cite{BCS22} to lift $F$ to a skew product over a shift map, then using the invariance principle~\cite{AV-IP} we prove that, if the exponents are not bounded away from zero, there exists some $F$-invariant measures that are invariant by holonomies. We use some \emph{pinching} and \emph{twisting} conditions, that are $C^1$ open and $C^r, r> 1,$ dense, to conclude that this extra invariance is not possible.

\section{Preliminaries}\label{s-preliminar}

\subsection{Measure of maximal entropy for homoclinic classes}

In this section, we establish the notation and present previous results that will be used throughout the paper. 

\subsubsection{Hyperbolic Sets and Homoclinic Classes}\label{subsectionHomoclinicClass}

In this subsection, we revisit the concept of homoclinic classes, originally introduced by Newhouse \cite{New72}.

Let $f: M \to M$ be a diffeomorphism defined on a compact manifold $M$. 
An $f$-invariant  compact set  $\Lambda$ is \textit{hyperbolic} for $f$ if there exists a splitting of the tangent bundle $T\Lambda$ into two subbundles $E^s$ and $E^u$ satisfying the following conditions:
\begin{enumerate}
    \item For every $x \in \Lambda$, the subbundles are invariant under $Df_x$, i.e.,
    $$
    Df_xE^s_x = E^s_{f(x)} \quad \text{and} \quad Df_xE^u_x = E^u_{f(x)}.
    $$
    \item There exist constants $C > 0$ and $0 < \lambda < 1$ such that for all $x \in \Lambda$:
    $$
    \|Df^n_x v\| \leq C \lambda^n \|v\|, \quad \text{for all } v \in E^s_x \text{ and } n \geq 1,
    $$
    and
    $$
    \|Df^{-n}_x v\| \leq C \lambda^{n} \|v\|, \quad \text{for all } v \in E^u_x \text{ and } n \geq 1.
    $$
\end{enumerate}

It is well known that if $\Lambda$ is a hyperbolic set, then for any $x \in \Lambda$, the following sets are $C^1$-immersed submanifolds \cite{HPS70}:
$$
W^s(x) = \{y \in M : \lim_{n \to \infty} d(f^n(x), f^n(y)) = 0\},
$$
$$
W^u(x) = \{y \in M : \lim_{n \to \infty} d(f^{-n}(x), f^{-n}(y)) = 0\}.
$$
These are referred to as the \textit{stable} and \textit{unstable manifolds} of $x$, respectively. 

When a periodic orbit satisfies the hyperbolicity condition it is called a \textit{hyperbolic periodic orbit}. The set of all hyperbolic periodic orbits of $f$ is denoted by $\operatorname{Per}_h(f)$.  For each $\cO\in \Per_h(f)$ we define $W^{s/u}(\cO):=\cup_{x\in \cO}W^{s/u}(x)$. 

For any $C^r$-submanifolds $U, V \subset M$, $r\geq 1$, their transverse intersection is defined as:
$$
U \pitchfork V = \{x \in U \cap V : T_xU + T_xV = T_xM\}.
$$

Two orbits $\mathcal{O}_1, \mathcal{O}_2 \in \operatorname{Per}_h(f)$ are  \textit{homoclinically related} if:
$$
W^s(\mathcal{O}_1) \pitchfork W^u(\mathcal{O}_2) \neq \emptyset \quad \text{and} \quad W^s(\mathcal{O}_2) \pitchfork W^u(\mathcal{O}_1) \neq \emptyset.
$$
In this case, we write $\mathcal{O}_1 \sim \mathcal{O}_2$. The pair $(\operatorname{Per}_h(f), \sim)$ forms an equivalence relation.
Given a subset $A\subset M$, $\overline{A}$ means the closure of $A$ in $M$.

\subsubsection{Hyperbolic Measures and Homoclinic Classes}\label{subsectionHyperbolicMeasures}

This subsection reviews key results concerning hyperbolic measures and homoclinic classes, following \cite{BCS22}.

Let $x \in M$ and let $v \in T_xM$ be a nonzero vector. The Lyapunov exponent of $f$ at $x$ in the direction of $v$ is defined as:
$$
\lambda(f,x,v) = \lim_{{n \to +\infty}} \frac{1}{n} \log \left\| Df^n(v) \right\|,
$$
provided the limit exists.

Let $\cM_f^e$ denote the set of ergodic measures for $f$. If $\mu \in \cM_f^e$, the Oseledets theorem guarantees the existence of Lyapunov exponents for $\mu$--almost every $x \in M$ \cite{Ose68}.

\begin{theorem}[Oseledets]
Let $(f, \mu)$ be as described above. Then, for $\mu$-almost every $x \in M$, there exist real numbers $\lambda_1(f,\mu) < \lambda_2(f,\mu) < \cdots < \lambda_k(f,\mu)$ and a splitting:
$$
T_xM = E_{x,1} \oplus E_{x,2} \oplus \cdots \oplus E_{x,k},
$$
such that:
\begin{enumerate}
    \item $Df(E_{x,i}) = E_{f(x),i}$ for $i = 1, \ldots, k$.
    \item For any nonzero $v \in E_{x,i}$, 
    $$
    \lambda(x,v) = \lim_{{n \to \pm\infty}} \frac{1}{n} \log \left\| Df^n_xv \right\| = \lambda_i.
    $$
\end{enumerate}
\end{theorem}

The set of points where the Oseledets theorem holds is known as the Lyapunov regular set, denoted by ${\cR}$. For simplicity, we write $\lambda_i(\mu):=\lambda_i(f,\mu)$ when there is no confusion on what diffeomorphism $f$ we are talking about.

\begin{definition}\label{Hyperbolicmeasure} An ergodic measure $\mu$ is \textit{hyperbolic} if all  its Lyapunov exponents satisfy $|\lambda_i(f,\mu)| > 0$. For a given $\alpha > 0$, an ergodic measure $\mu$ is  \textit{$\alpha$--hyperbolic} if all  its Lyapunov exponents satisfy $|\lambda_i(f,\mu)| > \alpha$. The sets of all ergodic hyperbolic measures and $\alpha$-hyperbolic ergodic measures are denoted by $\cM_f^h$ and $\cM_f^\alpha$, respectively. 
\end{definition}

 For $\mu \in \cM_f^h$, the tangent space at $\mu$-almost every $x \in M$ admits a splitting:
$$
T_x M = E^-_x(f) \oplus E^+_x(f),
$$
where $E^-_x$ and $E^+_x$ are subspaces satisfying:
$$
\lim _{n \to \infty} \frac{1}{n} \log \left\|D f_x^n v\right\| < 0 \quad \text{on } E^-_x \setminus \{0\},
$$
and
$$
\lim _{n \to \infty} \frac{1}{n} \log \left\|D f_x^{n} v\right\| > 0 \quad \text{on } E^+_x \setminus \{0\}.
$$

For $f \in \Diff^{1+}(M)$, the stable Pesin set of $x \in M$ is defined as:
$$
W^{s}(x)=\left\{y \in M: \limsup _{n \to +\infty} \frac{1}{n} \log d\left(f^n(x), f^n(y)\right) < 0\right\}.
$$
Similarly, the unstable Pesin set is:
$$
W^{u}(x)=\left\{y \in M: \limsup _{n \to +\infty} \frac{1}{n} \log d\left(f^{-n}(x), f^{-n}(y)\right) < 0\right\}.
$$
For $\mu \in \cM_f^h$, it is known that, for $\mu$-almost every $x \in M$, the sets $W^{s}(x)$ and $W^{u}(x)$ are immersed submanifolds of dimensions $\dim(E^-)$ and $\dim(E^+)$, respectively. These are referred to as the \textit{stable} and \textit{unstable Pesin manifolds} \cite{Pes76}. For hyperbolic periodic orbits, stable and unstable Pesin manifolds coincide with the stable and unstable manifolds. 

For  $\varepsilon > 0$, the connected components containing $x$ within the intersections $W^s(x) \cap B_\varepsilon(x)$ and $W^u(x) \cap B_\varepsilon(x)$ are called the \textit{local stable manifold} and \textit{local unstable manifold}, denoted by $W^s_\varepsilon(x)$ and $W^u_\varepsilon(x)$, respectively.

We will need a more quantitative definition of Pesin manifolds. Fixed $C>1$, $\chi>0$ and $\varepsilon>0$, define 
$$W^s_{C,\chi,\varepsilon}(f,x)=\{y\in B_\varepsilon(x);d(f^n(x),f^n(y)\leq Ce^{-\chi n}d(x,y),n\geq 0\}.$$
Analogously we define $W^u_{C,\chi,\varepsilon}(f,x)$ replacing $ n\geq 0$ by  $n\leq 0$. By \cite{Pes76} if $\mu\in \mathcal{M}^h_f$, for $\mu$ almost every $x\in M$, there exists $C,\chi$ such that $W^*_\varepsilon(f,x)=W^*_{C,\chi,\varepsilon}(f,x)$, $*=s,u$. For simplicity, when there is no confusion on which $f$ we are talking about we omit the variable $f$.

\begin{definition}\label{def-index}
    When a point $x \in M$ has well-defined stable and unstable Pesin manifolds, we say that its index is given by $\operatorname{index}(x) = \dim (T_x W^s(x))$.
The \textit{index} of a hyperbolic ergodic measure $\mu$ is defined as $\operatorname{index}(\mu) := \dim(E^-)$. 
\end{definition}

\begin{definition}
For two measures $\mu_1, \mu_2 \in \cM_f^h$, we write $\mu_1 \preceq \mu_2$ if there exist sets $\Lambda_1, \Lambda_2$ such that $\mu_i(\Lambda_i) > 0$, and for any $(x, y) \in \Lambda_1 \times \Lambda_2$, the stable and unstable Pesin manifolds satisfy:
$$
W^{u}(x) \pitchfork W^{s}(y) \neq \emptyset.
$$
\end{definition}

\begin{definition}\label{def-HR-measures}
Two ergodic measures $\mu_1, \mu_2 \in \cM_f^h$ are \textit{homoclinically related} if $\mu_1 \preceq \mu_2$ and $\mu_2 \preceq \mu_1$. In this case, we write $\mu_1 \sim \mu_2$. The homoclinic class of a measure $\mu \in \cM_f^h$ is defined as:
$$
HC(\mu) = \{\nu \in \cM_f^h: \mu \sim \nu\}.
$$
\end{definition}

The relation $\sim$ is an equivalence relation on $\cM_f^h$ \cite[Proposition 2.11]{BCS22}.

\begin{definition}\label{RelationMuO}
For a hyperbolic measure $\mu \in \cM_f^h$ and a hyperbolic periodic orbit $\mathcal{O} = \{p, f(p), \ldots, f^{n-1}(p)\}$, we say that $\mu$ and $\mathcal{O}$ are \textit{homoclinically related}, and write $\mathcal{O} \sim \mu$ or $p \sim \mu$, if the measure $\frac{1}{n} \sum_{j=0}^{n-1} \delta_{f^j(p)}$ is homoclinically related to $\mu$.
\end{definition}

\begin{proposition}[{\cite[Corollary 2.14]{BCS22}}]\label{HoclinicClassRelation}
For any $f \in \Diff^r(M)$, $r > 1$, and $\mu \in \cM_f^h$, the set of measures supported by periodic orbits is dense in the set of hyperbolic ergodic measures homoclinically related to $\mu$, in the weak-$\ast$ topology. In particular, there exists $\mathcal{O} \in \operatorname{Per}_h(f)$ such that $\mathcal{O} \sim \mu$ and:
$$
HC(\mu) = HC(\mathcal{O}).
$$
\end{proposition}

\begin{proposition}[{\cite{BCS22}}]\label{BCS-critirium}
Let $f: M \to M$ be a $C^r$ diffeomorphism with $r > 1$,  and $p$ a hyperbolic periodic point with orbit $\mathcal{O}$. Then, there is at most one hyperbolic measure of maximal entropy homoclinically related to $\mathcal{O}$, and its support coincides with $HC(\mathcal{O})$.
\end{proposition}


%
%

\subsection{Unstable entropy}

 \subsubsection{Unstable metric entropy}
 
Consider a diffeomorphism $f$ partially hyperbolic with $\dime(E^{uu})\geq 1$. Let $ \alpha $ be a partition of $ M $. We denote $ \alpha(x) $ as the element of $ \alpha $ containing the point $ x $. If we have two partitions $ \alpha $ and $ \beta $ such that $ \alpha(x) \subset \beta(x) $ for all $ x \in M $, we can write $ \alpha \geq \beta $ or $ \beta \leq \alpha $. If a partition $ \eta $ satisfies $ f^{-1}(\eta) \geq \eta $, we say that the partition is increasing.

A measurable partition $ \eta $ is  increasing and subordinate to $ \mathcal{F}^{uu} $ if it satisfies the following conditions:
\begin{itemize}
    \item[(a)] $ \eta(x) \subseteq \mathcal{F}^{uu}(x) $ for $\mu$-almost every $ x $.
    \item[(b)] $ f^{-1}(\eta) \geq \eta $ (increasing property).
    \item[(c)] $ \eta(x) $ contains an open neighborhood of $ x $ in $ \mathcal{F}^{uu}(x) $ for $\mu$-almost every $ x $.
\end{itemize}

The existence of measurable partitions that are increasing and subordinate to $ \mathcal{F}^{uu}$ was proved in \cite{LS82}.

\begin{definition}
    Consider a partially hyperbolic diffeomorphism $f$ such that  $\operatorname{dim}(E^{uu}) \geq 1 $, $\mu \in \mathcal{M}^e_f$ , and let $\eta$  be an increasing measurable partition subordinate to $ \cF^{uu} $. We define the unstable metric entropy of $ \mu $ as

    $$h^u_\mu(f) = H(f^{-1}\eta\mid \eta),$$
\end{definition}

This definition does not depend on the specific increasing measurable partition subordinate to $ \cF^{uu} $. For details, see \cite{LS82, HHW17, Yan21}.

\subsubsection{Unstable Topological Entropy and the Variational Principle}\label{ss.unstable-entropy}

We now introduce the concept of unstable topological entropy, first defined in \cite{SX10}. This quantity measures the dynamics complexity  along the unstable manifold.

Let $d^u$ be the metric induced by the Riemannian structure on the unstable manifold $W^{uu}$, and define

$$d_n^u(x, y) = \max_{0 \leq j \leq n-1} d^u\left(f^j(x), f^j(y)\right).$$

Given a point $x \in M$ and $\delta > 0$, let $W^{uu}(x, \delta)$ denote the open ball of radius $\delta$ centered at $x$ inside $W^{uu}(x)$ with respect to the metric $d^u$. Define $N^u(f, \epsilon, n, x, \delta)$ as the maximal number of points in $\overline{W^{uu}(x, \delta)}$ whose pairwise $d_n^u$-distances are at least $\epsilon$. Such a collection of points is called an \emph{$(n, \epsilon)$ u-separated set} of $\overline{W^{uu}(x, \delta)}$.
\begin{definition}
     The unstable topological entropy of $f$ on $M$ is given by:
\begin{equation*}
h^u(f) = \lim_{\delta \to 0} \sup_{x \in M} h^u\left(f, \overline{W^{uu}(x, \delta)}\right),
\end{equation*}
where
\begin{equation*}
h^u\left(f, \overline{W^{uu}(x, \delta)}\right) = \lim_{\epsilon \to 0} \limsup_{n \to \infty} \frac{1}{n} \log N^u(f, \epsilon, n, x, \delta).
\end{equation*}

\end{definition}

 Unstable topological entropy can also be defined equivalently, using $(n, \epsilon)$ u-spanning sets or open covers \cite{HHW17}.

 If the dimension of the stable bundle $E^{ss}_f$ is at least one, the \emph{stable entropy} is defined analogously as:
\[
h^s(f) := h^u(f^{-1}).
\]

 The unstable topological entropy relates to the metric unstable entropy via a variational principle. Specifically, \cite[Theorem D]{HHW17} states that for any $C^1$-partially hyperbolic diffeomorphism $f: M \to M$, we have:
\[
h^u(f) = \sup\{h_{\mu}^u(f) : \mu \in \mathcal{M}_f(M)\}
\quad \text{and} \quad 
h^u(f) = \sup\{h^u_{\nu}(f) : \nu \in \mathcal{M}^e_{f}(M)\}.
\]

 An alternative characterization of the unstable topological entropy involves the unstable volume growth, as studied by Hua, Saghin, and Xia in \cite{HSX08}, inspired by earlier work of Yomdin and Newhouse \cite{Yom87,New89}. It was shown in \cite[Theorem C]{HHW17} that the unstable topological entropy, as defined above, coincides with the unstable volume growth.

\subsection{Existence of measures for maximal Entropy}\label{subsebtion-dimEc1-mme}

Newhouse demonstrated that all $C^\infty$ diffeomorphisms possess a measure of maximal entropy \cite{New91}. However, as Misurewicz showed \cite{Mis73}, this property does not hold for diffeomorphisms with lower differentiability. While it is well-established that Anosov diffeomorphisms always have MME, partially hyperbolic systems, in general, may fail to admit MME \cite{B14}.

Let $f: M \to M$ be a homeomorphism defined on a compact metric space $M$. For a given $x \in M$ and $\varepsilon > 0$, define:
\[
\Gamma_\varepsilon(x) = \{y \in M : d(f^n(y), f^n(x)) < \varepsilon \text{ for all } n \in \mathbb{Z}\}.
\]

We say that $f$ is expansive at scale $\varepsilon$ if for every $x \in M$, $\Gamma_\varepsilon(x) = \{x\}$. Similarly, $f$ is entropy-expansive at scale $\varepsilon$ if for every $x \in M$, the topological entropy of $f$ restricted to $\Gamma_\varepsilon(x)$ satisfies:
$$
h(f|_{\Gamma_\varepsilon(x)}) = 0.
$$

Now, consider the function $h_{\cdot}(f)$ defined as:
$$
h_{\cdot}(f): \mathcal{M}_f \to \mathbb{R}, \quad \mu \mapsto h_{\mu}(f),
$$
where $\mathcal{M}_f$ denotes the space of $f$-invariant probability measures. Since $\mathcal{M}_f$ is compact in the weak-* topology, the upper semicontinuity of the metric entropy function ensures that $h_{\cdot}(f)$ attains a maximum. 

For homeomorphisms, Bowen proved in \cite{Bow72} that the metric entropy is upper semicontinuous whenever $(M, f)$ is an entropy-expansive system. Consequently, $f$ admits MMEs. 


\section{Invariant Manifolds with Uniform size}\label{section-PesinBlock}

Consider a diffeomorphism $f:M\to M$ with a dominated splitting $T M = E^{cs} \oplus E^{cu}$. For each $\ell \in \mathbb{Z}_+$ we consider  the set of points $x\in M$, $B_{\ell}^s(f,E^{cs})$, for which the following inequality holds:

$$\frac{1}{n} \sum_{j=0}^{n-1} \log \| Df^\ell (f^{\ell j}(x)) | E^{cs} \| < -1 \quad \text{for every } n \in \mathbb{Z}_+.$$

Similarly, define $B_{\ell}^u(f,E^{cu})$ as the set  of points $x\in M$ such that 

$$\frac{1}{n} \sum_{j=0}^{n-1} \log \| Df^{-\ell} (f^{-\ell j}(x)) | E^{cu} \| < -1 \quad \text{for every } n \in \mathbb{Z}_+.
$$
We define the $\ell$-Pesin block as

$$B_{\ell}(f,E^{cs},E^{cu}) = B_{\ell}^s(f,E^{cs}) \cap B_{\ell}^u(f,E^{cu}).$$
For simplicity, when there is no confusion on which decomposition we are dealing with, we write only $B^*_\ell(f)$.
Note that $B_{\ell}^u(f)=B_{\ell}^s(f^{-1})$.

\begin{theorem}\label{Invariant-manifolds}
    Consider a diffeomorphism $f: M \to M$ with a dominated splitting $T M = E^{cs} \oplus E^{cu}$. For each $\ell \in \mathbb{Z}_+$, there exist constants $r = r(\ell) > 0$, $C=C(\ell)$, $\chi=\chi(\ell)$, $\varepsilon=\varepsilon(\ell)$ and $\delta = \delta(\ell) > 0$ such that:
    \begin{enumerate}
        \item If $x \in B_{\ell}(f)$, then $W^s_{C,\chi,\varepsilon}(f, x)$ and $W^u_{C,\chi,\varepsilon}(f, x)$ are disks of radius bigger than $r$ centered at $x$.
        \item If $x, y \in B_{\ell}(f)$ and $d(x, y) < \delta$, then $W^s_{C,\chi,\varepsilon}(f, x)$ and $W^u_{C,\chi,\varepsilon}(f, x)$ intersect transversely.
    \end{enumerate}
    
    Moreover, the same constants $r,C,\chi,\varepsilon$ applies to any diffeomorphism $g$ that is sufficiently $C^1$-close to $f$ (depending on $\ell$). Additionally, $W^s_{C,\chi,\varepsilon}(g, x)$ and $W^u_{C,\chi,\varepsilon}(g, x)$ vary continuously with respect to $g$ in the $C^1$-topology and with respect to $x \in B_{\ell}(g)$ in the metric $d$.
\end{theorem}
\begin{proof}
    The proof is a consequence of \cite{HPS70}. See also \cite[Theorem 4.1]{AB12} and \cite[Theorem 3.5.]{AV20}
\end{proof}

\subsection{The Size of the Pesin Blocks for $\alpha$--hyperbolic measures}

We recall that  $$\lambda^{cs}(x) = \lim_{\ell \to +\infty} \frac{1}{\ell} \log \| Df^\ell | E^{cs}(x) \|$$
when the limit exists.

We set 

\begin{equation}\label{Hyperbolic-stable-points}
  \Lambda_{\alpha}(f,E^{cs}):=\Lambda^{cs}_\alpha(f):=\left\{ x\in M: \lambda^{cs}(x)\leq-\alpha\right\},  
\end{equation}
\begin{equation}\label{Hyperbolic-unstable-points}
  \Lambda_{\alpha}(f,E^{cu}):=\Lambda^{cu}_\alpha(f):=\left\{ x\in M: \lambda^{cu}(x)\geq\alpha\right\},  
\end{equation}

and
\begin{equation}\label{Hyperolic-points}
  \Lambda_{\alpha}(f,E^{cs},E^{cu}):=\Lambda_{\alpha}(f):=\Lambda^{cs}_\alpha(f) \cap \Lambda^{cu}_\alpha(f).  
\end{equation}

The next lemma says that if $\lambda^{cs}$ is negative for many points with respect  to $\mu$, then $B_{\ell}^s(f)$ has large measure when $\ell$ is large enough. This Lemma is similar to \cite[Lemma 4.8]{AB12}

\begin{lemma}\label{better-Size-of-PesinBlocks}
Let $\mu \in \mathcal{M}_f$. Assume that $\eta > 0$, $\alpha > 0$, $b>0$ and $\ell \in \mathbb{Z}_+$ satisfy the following conditions:

\begin{equation}\label{C1}
\mu \left\{ x \in M : \lambda^{cs}(x) > -\alpha \right\} < \eta,    
\end{equation}

\begin{equation}\label{C2}
    \ell > 1/(\alpha b),
\end{equation}
    

    \begin{equation}\label{C3}
\int_{\left\lbrace \frac{1}{\ell} \log \| Df^\ell | E^{cs}\| >-\alpha\right\rbrace \cap \left\lbrace \lambda^{cs} \leq -\alpha\right\rbrace}\left( \frac{1}{\ell} \log \| Df^\ell | E^{cs}\| +\alpha\right) d\mu <  b\alpha.    
\end{equation}

Then:
$$\mu(M\setminus B_{\ell}^s(f)) < 2b + \eta.$$

\end{lemma}

\begin{proof}

Let $\eta > 0$, $\alpha > 0$, $b>0$ and $\ell \in \mathbb{Z}_+$ as in above Lemma. For every $x\in M$ we define
    \begin{equation}
        \varphi(x)=\log\|Df^\ell(x)|E^{cs}(x)\| \quad \mbox{ and } \quad \varphi^*(x)=\max_{n\geq 1}\frac{1}{n}\sum_{j=0}^{n-1}\varphi(f^{\ell j}(x).
    \end{equation}
Thus, $B^s_\ell(f)=\{x\in M :\varphi^*(x)<-1\}$.

 By the Maximal Ergodic Theorem applied to $f^\ell$ in the invariant set $\{x\in M: \lambda^{cs}(x)\leq-\alpha\}$ we get that

$$
\int_{\left\{\varphi^{*} \geq-1\right\} \cap\left\{\lambda^{c s} \leq-\alpha\right\}}(\varphi+1) d \mu \geq 0 .
$$

Therefore,
$$
\begin{aligned}
0 & \leq \int_{\left\{\varphi^{*} \geq-1\right\} \cap\left\{\lambda^{c s} \leq-\alpha\right\}} \frac{\varphi+1}{\ell} d \mu  \\
& \leq b\alpha+\int_{\left\{\varphi^{*} \geq-1\right\} \cap\left\{\lambda^{c s} \leq-\alpha\right\}} \frac{\varphi}{\ell} d \mu  \mbox{      (by \eqref{C2})}\\
& \leq b \alpha+\int_{\left\{\varphi^{*} \geq-1\right\} \cap\left\{\lambda^{c s} \leq-\alpha\right\}} (\frac{\varphi}{\ell}+\alpha -\alpha) d \mu  \\
& \leq b\alpha+\int_{\left\{\varphi^{*} \geq-1\right\} \cap\left\{\lambda^{c s} \leq-\alpha\right\}} -\alpha d \mu+\int_{\left\{\varphi/\ell+\alpha> 0\right\}\cap \left\{\lambda^{cs}\leq -\alpha\right\}} (\frac{\varphi}{\ell}+\alpha) d\mu  \mbox{     (by \eqref{C3})}\\
& \leq 2b\alpha-\alpha \mu\left(\left\{\varphi^{*} \geq-1\right\} \cap\left\{\lambda^{c s} \leq-\alpha\right\}\right).
\end{aligned}
$$

Thus,  $\mu\left(\left\{\varphi^{*} \geq-1\right\} \cap\left\{\lambda^{c s} \leq-\alpha\right\}\right) \leq 2b$ and applying  \eqref{C1} we get that

$$\mu(\left\{\varphi^{*} \geq-1\right\})< 2b + \eta,$$  as we wanted to show.
    
\end{proof}

\begin{lemma}\label{Uniform Pesin Blocks}
    Let $f:M\to M$ and $f_n: M \to M$ be diffeomorphisms of a compact manifold $M$ admitting a dominated splitting $E^{cs} \oplus E^{cu}$, where $f_n$ converges to $f$ in the $C^1$ topology. Let $\alpha > 0$ and suppose that $\{\mu_n\}_{n=1}^\infty \subset \cM_{f_n}^e$ is a sequence of $\alpha$-hyperbolic measures with $\operatorname{Index}(\mu_n) = \dim(E^{cs})$, converging to $\mu \in \cM_f$ such that $\mu(\Lambda^{cs}_\alpha(f)) = 1$. Then, for any $\varepsilon > 0$, there exist constants $\ell', N \in \mathbb{Z}^+$ such that for all $n \geq N$ and $\ell \geq \ell'$, both $\mu_n(B_{\ell}^s(f))$ and $\mu(B_{\ell}^s(f))$ are greater than $1 - \varepsilon$.
\end{lemma}

\begin{proof}
    Consider the sequence of measures $\{\mu_n\}_{n=1}^\infty$ converging to $\mu$ as specified in the statement of the lemma, and let $\varepsilon > 0$.

    Since $\mu(\Lambda^{cs}_\alpha(f)) = 1$, we have:
    \begin{equation}\label{hold-C1}
        \mu \left( \left\{ x \in M : \lambda^{cs}(x) > -\alpha \right\} \right) = 0 < \frac{\varepsilon}{2}.
    \end{equation}

    Additionally, there exists an integer $\ell'$ such that for every $\ell \geq \ell'$:
    \begin{equation}\label{LEclose}
        \int \left| \frac{1}{\ell} \log \| Df^\ell | E^{cs} \| - \lambda^{cs} \right| d\mu < \frac{\varepsilon}{4}.
    \end{equation}

    We fix $\ell'' > \ell'$ such that :
    \begin{equation}\label{hold-C2}
        \ell'' > \frac{1}{\frac{\varepsilon}{4}}.
    \end{equation}

    Now, observe that for $\mu$-a.e. $x\in M$:
    $$
    \begin{aligned}
        \frac{1}{\ell} \log \| Df^\ell | E^{cs} \| + \alpha &= \frac{1}{\ell} \log \| Df^\ell | E^{cs} \| - \lambda^{cs} + \lambda^{cs} + \alpha \\
        &\leq \frac{1}{\ell} \log \| Df^\ell | E^{cs} \| - \lambda^{cs}, \quad \text{since } \lambda^{cs} + \alpha \leq 0.
    \end{aligned}
    $$

    Therefore, by \eqref{LEclose}, for each $\ell>\ell'$ we have
    \begin{equation}\label{hold-C3}
        \int_{\left\lbrace \frac{1}{\ell} \log \| Df^\ell | E^{cs} \| > -\alpha \right\rbrace \cap \left\lbrace \lambda^{cs} \leq -\alpha \right\rbrace} \left( \frac{1}{\ell} \log \| Df^\ell | E^{cs} \| + \alpha \right) d\mu < \frac{\varepsilon}{4}.
    \end{equation}

    Now, take $\eta = \frac{\varepsilon}{2}$ and $b = \frac{\varepsilon}{4}$ in Lemma \ref{better-Size-of-PesinBlocks}. By \eqref{hold-C1}, \eqref{hold-C2}, and \eqref{hold-C3}, we conclude that for each $\ell>\ell''$ it holds
    $$
    \mu(B_\ell^s(f)) > 1 - \varepsilon.
    $$

    Since $\mu_n(\Lambda^{cs}_\alpha(f))=1$ and $\mu(\Lambda^{cs}_\alpha(f))=1$, we also have:
    \begin{equation}\label{hold-C1'}
        \mu_n \left( \left\{ x \in M : \lambda^{cs}(x) > -\alpha \right\} \right) = 0 < \frac{\varepsilon}{2}.
    \end{equation}


    Define the continuous map $\tilde{\varphi}: M \to \mathbb{R}$ by:
    
    $$\tilde{\varphi}(x) =\max\{\frac{1}{\ell} \log \| Df^\ell | E^{cs} \| + \alpha,0\}$$
    and 
    $$\tilde{\varphi}_n(x) =\max\{\frac{1}{\ell} \log \| Df_n^\ell | E^{cs} \| + \alpha,0\}.$$
    
 Using again that  $\mu(\Lambda^{cs}_\alpha(f))=1$ and $\mu_n(\Lambda^{cs}_\alpha(f))=1$, we get that for $\mu$ and each $\mu_n$, the sets $\left\lbrace \frac{1}{\ell} \log \| Df^\ell | E^{cs} \| > -\alpha \right\rbrace \cap \left\lbrace \lambda^{cs} \leq -\alpha \right\rbrace$ and $\left\lbrace \frac{1}{\ell} \log \| Df^\ell | E^{cs} \| > -\alpha \right\rbrace$ have the same measure, which implies that
    $$ 
    \int_{\left\lbrace \frac{1}{\ell} \log \| Df^\ell | E^{cs} \| > -\alpha \right\rbrace \cap \left\lbrace \lambda^{cs} \leq -\alpha \right\rbrace} \left( \frac{1}{\ell} \log \| Df^\ell | E^{cs} \| + \alpha \right) d\mu = \int \tilde{\varphi} \, d\mu,
    $$
    and similarly,
    $$
    \int_{\left\lbrace \frac{1}{\ell} \log \| Df^\ell | E^{cs} \| > -\alpha \right\rbrace \cap \left\lbrace \lambda^{cs} \leq -\alpha \right\rbrace} \left( \frac{1}{\ell} \log \| Df^\ell | E^{cs} \| + \alpha \right) d\mu_n = \int \tilde{\varphi}_n \, d\mu_n.
    $$

    Since $\mu_n$ converges to $\mu$ in the weak$^*$ topology and $Df^\ell_n|E^{cs}$ converges to $Df^\ell|E^{cs}$ in the $C^0$ topology, there exists $N > 0$ such that for every $n \geq N$,
    \begin{equation}\label{hold-c3'}
        \int_{\left\lbrace \frac{1}{\ell} \log \| Df^\ell | E^{cs} \| > -\alpha \right\rbrace \cap \left\lbrace \lambda^{cs} \leq -\alpha \right\rbrace} \left( \frac{1}{\ell} \log \| Df^\ell | E^{cs} \| + \alpha \right) d\mu_n < \frac{\varepsilon}{4}.
    \end{equation}

    Thus, by Lemma \ref{better-Size-of-PesinBlocks}, together with \eqref{hold-C1}, \eqref{hold-C1'}, and \eqref{hold-c3'}, we conclude that for every $n \geq N$ and $\ell>\ell''$,
    $$
    \mu_n(B_\ell^s(f)) > 1 - \varepsilon,
    $$
    as required.
\end{proof}

\section{Continuity of Entropy}

Let $ p \in M $ be a hyperbolic periodic point of a diffeomorphism $f: M \to M $. A point $ q \in W^s(p) \cap W^u(p)$, distinct from $p$, is called a homoclinic point associated with $p$. The homoclinic point $q$ is said to be transverse if  
$$
T_qM = T_qW^u(p) + T_qW^s(p).
$$  
In this case, we denote it as $q \in W^s(p) \pitchfork W^u(p)$.

We say that $f$ has a homoclinic tangency if there exists a non-transverse homoclinic point associated to some hyperbolic periodic point. The set of $C^r$ diffeomorphisms that have some homoclinic tangency will be denoted by $\operatorname{HT}^r(M)$. We say that $f\in \Diff^1(M)$ is away from tangencies if $f\in \Diff^{1}(M)\setminus \overline{\operatorname{HT^1(M)}}$. 

\begin{remark}\label{f is away from tangeccies}
    Note that if  $f: M \to M$ is a diffeomorphism on a compact manifold $M$ admitting a dominated splitting $E^{ss} \oplus E_1 \oplus \cdots \oplus E_k \oplus E^{uu}$, where each $E_i$ is one-dimensional,  then $f\in \Diff^{1}\setminus \overline{\operatorname{HT^1}(M)}$.
\end{remark}
\begin{proposition}\label{upper-semicontinuous, entropy expansive}

Let $f_n: M \to M$, $n\in \mathbb{N}$, be a sequence of homeomorphisms converging to $f: M \to M$, where $M$ is a compact metric space, and suppose each $f_n$ and $f$ are entropy expansive at the same scale $\delta > 0$. If $\mu_n \subset \cM_{f_n}$ converges to $\mu \in \cM_f$, then:
$$
\limsup_{n \to \infty} h_{\mu_n}(f_n) \leq h_\mu(f).
$$
\end{proposition}

The proof is similar to the proof of \cite[Theorem 8.2]{walters2000}.

\begin{proof} Fix $\varepsilon > 0$ and let $\delta>0$ be  the common scale of entropy expansiveness for each $f_n$ and $f$. Consider a partition $\gamma = \{C_1, \ldots, C_k\}$ of $M$ into Borel sets with $\operatorname{diam}(C_j) < \delta$. By \cite[Theorem 3.5]{Bow72}, we know that $h_\mu(f) = h_\mu(f, \gamma)$. Therefore, we can choose $N > 0$ large enough such that:
$$
\frac{1}{N} H_\mu\left(\bigvee_{j=0}^{N-1} f^{-j} \gamma\right) < h_\mu(f) + \frac{\varepsilon}{2}.
$$

For each $i_0, \ldots, i_{N-1}$, we select compact sets $K(i_0, \ldots, i_{N-1}) \subset \bigcap_{j=0}^{N-1} f^{-j} C_{i_j}$ such that:
\begin{equation}\label{compact-set-choice}
\mu\left(\bigcap_{j=0}^{N-1} f^{-j} C_{i_j} \setminus K(i_0, \ldots, i_{N-1})\right) < \varepsilon_1,
\end{equation}
where   $\varepsilon_1 > 0$ is a fixed positive number that will be chosen later.

For each $i \in \{1, \ldots, k\}$, define:
$$
L_i := \bigcup_{j=0}^{N-1} \{f^j K(i_0, \ldots, i_{N-1})|i_j=i\}.
$$
Note that  $L_i \subset C_i$ and the sets $L_1, \ldots, L_k$ are compact and pairwise disjoint, so we can define the partition $\gamma' = \{C_1', \ldots, C_k'\}$, with $\operatorname{diam}(C_i') < \delta$ and $L_i \subset \operatorname{int}(C_i')$. Thus, for all $i_0, \ldots, i_{N-1}$:
$$
K(i_0, \ldots, i_{N-1}) \subset \operatorname{int}\left(\bigcap_{j=0}^{N-1} f^{-j} C_{i_j}'\right).
$$

Since $f_n \to f$ in the $C^0$ topology, there exists $N_1 \in \mathbb{N}$ such that for all $n \geq N_1$:
$$
K(i_0, \ldots, i_{N-1}) \subset \operatorname{int}\left(\bigcap_{j=0}^{N-1} f_n^{-j} C_{i_j}'\right).
$$

By Urysohn's lemma, we construct continuous functions $\varphi_{i_0, \ldots, i_{N-1}}: M \to \mathbb{R}$, with $0 \leq \varphi_{i_0, \ldots, i_{N-1}} \leq 1$, such that $\varphi_{i_0, \ldots, i_{N-1}} = 1$ on $K(i_0, \ldots, i_{N-1})$ and $\varphi_{i_0, \ldots, i_{N-1}} = 0$ outside $\operatorname{int}\left(\bigcap_{j=0}^{N-1} f^{-j} C_{i_j}'\right)$. Here, if $A\subset M$,  $\operatorname{int}(A)$ means the interior of $A$.

Now, we define the weak* neighborhood of $\mu$:
$$
U(i_0, \ldots, i_{N-1}) = \left\{\nu \in \cM(M) : \left|\int \varphi_{i_0, \ldots, i_{N-1}} d\nu - \int \varphi_{i_0, \ldots, i_{N-1}} d\mu \right| < \varepsilon_1 \right\}.
$$

If $\nu \in U(i_0, \ldots, i_{N-1})$, then:
$$
\nu\left(\bigcap_{j=0}^{N-1} f_n^{-j} C_{i_j}'\right) >\mu(K(i_0, \ldots, i_{N-1})) - \varepsilon_1.
$$

So, if $\nu \in U(i_0, \ldots, i_{N-1}) $, then   \eqref{compact-set-choice} implies:
$$
\nu\left(\bigcap_{j=0}^{N-1} f_n^{-j} C_{i_j}'\right) - \mu\left(\bigcap_{j=0}^{N-1} f^{-j} C_{i_j}\right) < 2\varepsilon_1.
$$

Taking $U = \bigcap_{i_0, \ldots, i_{N-1}} U(i_0, \ldots, i_{N-1})$, if $\nu \in U$, as $\nu$ and $\mu$ are probability measures we have that 

$$
\left|\nu\left(\bigcap_{j=0}^{N-1} f_n^{-j} C_{i_j}'\right) - \mu\left(\bigcap_{j=0}^{N-1} f^{-j} C_{i_j}\right)\right| < 2\varepsilon_1,
$$

and using the continuity of $x \mapsto x \log x$, we can choose $\varepsilon_1$ such that for sufficiently large $n$:
$$
\frac{1}{N} H_{\mu_n}\left(\bigvee_{j=0}^{N-1} f_n^{-j} \gamma'\right) < \frac{1}{N} H_\mu\left(\bigvee_{j=0}^{N-1} f^{-j} \gamma\right) + \frac{\varepsilon}{2}.
$$

Thus, for large $n$:
$$
\begin{aligned}
h_{\mu_n}(f_n) & = h_{\mu_n}(f_n, \gamma') \quad \text{(by \cite[Theorem 3.5]{Bow72}, since $\operatorname{diam}(\gamma') < \delta$)} \\
& \leq \frac{1}{N} H_{\mu_n}\left(\bigvee_{j=0}^{N-1} f_n^{-j} \gamma'\right) \\
& < \frac{1}{N} H_\mu\left(\bigvee_{j=0}^{N-1} f^{-j} \gamma\right) + \frac{\varepsilon}{2} \\
& < h_\mu(f) + \varepsilon.
\end{aligned}
$$

Hence:
$$
\limsup_{n \to \infty} h_{\mu_n}(f_n) \leq h_\mu(f),
$$
as desired.
\end{proof}

\begin{theorem}\label{Upper semi continuity of metric entropy}
Let $f: M \to M$ be a diffeomorphism on a compact manifold $M$ admitting a dominated splitting $E^{ss} \oplus E_1 \oplus \cdots \oplus E_k \oplus E^{uu}$, where  each $E_i$ is one-dimensional, and let $f_n\subset \Diff^1(M)$, $n\in \mathbb{N}$, be a sequence converging to $f$ in the $C^1$ topology. Suppose $\mu_n \in \cM_{f_n}$ is a sequence of invariant measures converging to $\mu \in \cM_f$. Then, the following holds:
$$
\limsup_{n \to \infty} h_{\mu_n}(f_n) \leq h_\mu(f).
$$
\end{theorem}
\begin{proof}
    By Remark \ref{f is away from tangeccies} and \cite[Theorem 3.1]{LVY13}  there exists $\delta>0$ such that for $n$ big enough each $f_n$ is entropy expansive at scale $\delta$. So, by Proposition \ref{upper-semicontinuous, entropy expansive},    
 $$
\limsup_{n \to \infty} h_{\mu_n}(f_n) \leq h_\mu(f),
$$
as we wanted to prove.
\end{proof}

\begin{proposition}\label{countinuity of entropy}
    Let $f: M \to M$ be a diffeomorphism on a compact manifold $M$ admitting a dominated splitting $E^{ss} \oplus E_1 \oplus \cdots \oplus E_k \oplus E^{uu}$, where each $E_i$ is one-dimensional. If $f$ possesses a hyperbolic measure of maximal entropy  then the topological entropy is continuous at $f$. 
\end{proposition}
\begin{proof}
    Upper-semi-continuity follows from Theorem \ref{Upper semi continuity of metric entropy} while lower-semicontinuity follows from Katok's horseshoes  in \cite{K80, KH95} (see also \cite[Theorem 1]{Gel16}). 
\end{proof}

\section{Finite many measure of maximal entropy}\label{ProofTeoA}

\begin{lemma}\label{Uniform-size-implies-homoclinically -related}
    Let $f: M \to M$ be a $C^{1+}$ diffeomorphism with a dominated splitting $E^{cs} \oplus E^{cu}$, $\ell \in \mathbb{Z}^+$, and $\{\mu_n\}_{n \in \mathbb{N}}$ a sequence of hyperbolic measures of $f$ such that 
    $$\mu_n(B_\ell(f)) > 0 \quad \text{for all } n \in \mathbb{N}.$$ 
    Then, at least two of these measures are homoclinically related.
\end{lemma}

\begin{proof}
Let $\{\mu_n\}$ be a sequence of hyperbolic measures as in the above lemma and $\delta,r>0$ as in Theorem \ref{Invariant-manifolds}. For each $n \in \mathbb{N}$, since $\mu_n(B_{\ell}(f)) > 0$, we can choose $x_n \in B_{\ell}(f)$ such that 
\begin{equation}
\mu_n\left(B_{\delta/3}(x_n) \cap B_{\ell}(f)\right) > 0.
\end{equation}

Since $M$ is compact, there is a convergent subsequence $\{x_{n_k}\}$ such that $x_{n_k} \to x$ as $k \to \infty$. Consequently, there exists $N' \in \mathbb{N}$ such that for all $k_1, k_2 > N'$, we have $d(x_{k_1}, x_{k_2}) < \delta/3$. This implies that for any $z_1 \in B_{\delta/3}(x_{k_1}) \cap B_{\ell}(f)$ and $z_2 \in B_{\delta/3}(x_{k_2}) \cap B_{\ell}(f)$, the distance $d(z_1, z_2) < \delta$. Thus, Theorem~\ref{Invariant-manifolds} ensures that 
\begin{equation}
    W^{s/u}_r(z_1) \pitchfork W^{u/s}_r(z_2) \neq \emptyset,
\end{equation}
for all $z_1 \in B_{\delta/3}(x_{k_1}) \cap B_{\ell}(f)$ and $z_2 \in B_{\delta/3}(x_{k_2}) \cap B_{\ell}(f)$.

Therefore, we conclude that $\mu_{k_1} \sim \mu_{k_2}$ for all $k_1, k_2 > N'$, proving that at least two of these measures are homoclinically related.
\end{proof}

As a consequence, we get the following:

\begin{theorem}\label{MainTheoA}
Let $f: M \to M$ be a diffeomorphism on a compact manifold $M$ admitting a dominated splitting $E^{cs} \oplus E^{cu}$. Suppose there exist  a constant   $\alpha > 0$ such that for every $\mu \in \mathcal{M}^e_f$ with $h_\mu(f)=h(f) $ the following conditions hold:
\begin{enumerate}
    \item $\mu$ is $\alpha$--hyperbolic, and
    \item $\operatorname{index}(\mu) = \dim(E^{cs})$.
\end{enumerate}
Then $f$ admits only a finite number of ergodic MMEs.
\end{theorem}

\begin{proof}


The proof goes by contradiction. Assume that there exist infinitely many MMEs. So, we can take a sequence $\mu_n$ of different MMEs for $f$ which converges to $\mu \in \cM_f$. By hypotheses, each $\mu_n$ is $\alpha$--hyperbolic with $\operatorname{index}(\mu_n)=\dim{E^{cs}}$.

By  \cite[Theorem 1.2]{LSW15}, $\mu$ is a measure of maximal entropy, not necessarily ergodic. 
 By the Ergodic Decomposition Theorem, this implies that
\begin{equation}
    \mu(\Lambda^{cs}_{\alpha}(f, E^{cs}, E^{cu}))=1.
\end{equation}
Here $  \mu(\Lambda^{cs}_{\alpha}(f, E^{cs}, E^{cu}))$ is as in \eqref {Hyperolic-points}.

Fix $0<\varepsilon< 1/2$. By Lemma \ref{Uniform Pesin Blocks}, there exist integers $\ell'_1$ and $N_1$ such that for any $n \geq N_1$ and $\ell \geq \ell'_1$, both $\mu_n(B^s_{\ell}(f))$ and $\mu(B^s_{\ell}(f))$ are greater than $1-\varepsilon$. Similarly, by symmetry, there exist $\ell'_2$ and $N_2$ such that if $n \geq N_2$ and $\ell \geq \ell'_2$, then both $\mu_n(B^u_{\ell}(f))$ and $\mu(B^u_{\ell}(f))$ are greater than $1-\varepsilon$. Thus, for $\ell \geq \ell' := \max\{\ell'_1, \ell'_2\}$ and $n \geq \max\{N_1, N_2\}$, we have that both $\mu_n(B_{\ell}(f))$ and $\mu(B_{\ell}(f))$ are greater than $1-2\varepsilon$. 

Therefore, by Lemma \ref{Uniform-size-implies-homoclinically -related} there exist $k_1,k_1\in \bN$ such that $\mu_{k_1} \sim \mu_{k_2}$, which is a contradiction to Proposition \ref{BCS-critirium}. This concludes the proof.
\end{proof}

Let $f: M \to M$ be as described in Theorem \ref{theo-entropy-condition-dimension-2}. We denote the Lyapunov exponents in the directions $E_1$ and $E_2$ with respect to an ergodic measure $\mu$ of $f$ by $\lambda_1(f,\mu)$ and $\lambda_2(f,\mu)$, respectively.

\begin{lemma}\label{h(f)>h^s(f),h^u(f) implies far away from zero}
    Let $f: M \to M$ be as in Theorem \ref{theo-entropy-condition-dimension-2}. If $\mu \in \mathcal{M}^e_f$ satisfies $h_{\mu}(f) > \max\{h^s(f), h^u(f)\}$, then:
    $$
    \lambda_1(f, \mu) < h^s(f) - h_\mu(f) < 0 \quad \text{and} \quad \lambda_2(f, \mu) > h_\mu(f) - h^u(f) > 0.
    $$ 
\end{lemma}

\begin{proof}
    For $f \in \Diff^{1+}(M)$ and $\mu$ ergodic, by the Ledrappier–Young formula \cite{LY85a, LY85b} and \cite{Bro22}, we have:
    \begin{equation}\label{L-Yformula}
        h_\mu(f) \leq h_\mu^u(f) + \sum_{\lambda_c(f, \mu) > 0} \lambda_c(f, \mu).
    \end{equation}
    
    Therefore, if $\mu \in \mathcal{M}^e_f$ satisfies $h_\mu(f) > h^u(f)$, it follows that $\lambda_2(f, \mu) > 0$. Similarly, if $h_\mu(f) > h^s(f)$, then $\lambda_1(f, \mu) < 0$. 
    This implies that 
    $$
     h_\mu(f) \leq h_\mu^u(f) + \lambda_2(f, \mu) \quad \mbox{and}\quad 
     h_\mu(f) \geq h_\mu^s(f) + \lambda_1(f, \mu),
     $$
 finishing the proof.
\end{proof}
\begin{lemma}\label{Uniformly far away from zero}
    Let $f: M \to M$ be a diffeomorphism as in Theorem \ref{theo-entropy-condition-dimension-2}. Then, there exist constants $\alpha > 0$ and a $C^1$ neighborhood $\mathcal{U}$ of $f$ such that if $g \in \Diff^{1+}(M) \cap \mathcal{U}$ and $\mu \in \mathcal{M}^e_g$ is a measure of maximal entropy, then $\mu$ is an $\alpha$--hyperbolic measure.
\end{lemma}

\begin{proof}
    By Lemma \ref{h(f)>h^s(f),h^u(f) implies far away from zero}, every measure of maximal entropy for $f$ is hyperbolic. Moreover, by Proposition \ref{countinuity of entropy} and the upper semicontinuity of $u,s$-entropy, there exist a $C^1$ neighborhood $\mathcal{U}$ of $f$ and a constant $\alpha > 0$ such that for every $g \in \mathcal{U} \cap \Diff^{1+}(M)$, we have:
$$ h (g) - \max \{h^s (g), h^u (g)\} > \alpha.$$

    Therefore, applying Lemma \ref{h(f)>h^s(f),h^u(f) implies far away from zero}, it follows that for every ergodic measure of maximal entropy $\mu$ of $g \in \mathcal{U} \cap \Diff^{1+}(M)$, we have $|\lambda_i(f, \mu)| > \alpha$ for $i=1,2$. This concludes the proof.
\end{proof}

   \begin{lemma}\label{comparation metric entropy}
Consider $f_n:M\to M$ a sequence of $C^{1+}$ diffeomorphisms converging to $f$.  If $\mu_n \in \cM_{f_n}^e$ is an ergodic measure of maximal entropy for $f_n$  converging to $\mu\in \cM_f$ then $h(f)=h_\mu(f)$.
    \end{lemma}
\begin{proof}
    Since $h_{\mu_n}(f_n)=h(f_n)$, by Proposition \ref{countinuity of entropy} we have that 

    \begin{equation}\label{converging to entropy of f}
        \lim_nh_{\mu_n}(f_n)=h(f).
    \end{equation}

    Aditionally, by Proposition \ref{upper-semicontinuous, entropy expansive}
\begin{equation}
   \lim_nh_{\mu_n}(f_n)\leq h_\mu(f). 
\end{equation}
    So, by \eqref{converging to entropy of f} and \eqref{comparation metric entropy} we have that 

$$h_\mu(f)=h(f),$$
as required.
\end{proof}

We denote $E^{s+1}:=E^{ss}\oplus E_1$ and $E^{u+2}:= E_2\oplus E^{uu}$.

\begin{proof}[Proof of Theorem \ref{theo-entropy-condition-dimension-2}]
    Let $f:M\to M$ be a diffeomorphism as in Theorem \ref{theo-entropy-condition-dimension-2}. By Lemma \ref{Uniformly far away from zero} there exist $\alpha>0$ and a $C^1$ neighborhood of $f$ such that each ergodic measure of maximal entropy of a $C^r,  r>1$, diffeomorphism is $\alpha$--hyperbolic with $\operatorname{index}$ equal to $s+1$. Then, by  Theorem \ref{MainTheoA}, each $C^r$ diffeomorphisms into $\cU$ has finitely many measures of maximal entropy.  

    Let $k>0$ be the number of measures of maximal entropy of $f$. Let $\nu_1, \ldots, \nu_k$ be these ergodic measures of maximal entropy of $f$, and $p_1(f), \ldots, p_{k}(f)$ be periodic points of $f$ with orbits $\mathcal{O}_1\sim \nu_1, \ldots, \mathcal{O}_k\sim \nu_k$, given by Proposition \ref{HoclinicClassRelation}. 

    Consider a sequence $f_n$ of $C^r$ diffeomorphisms converging to $f$ in the $C^1$ topology.
    Let $p_i(f_n)$ be the unique hyperbolic periodic point of $f_n$ derived from the structural stability of $p_i(f)$.

    \begin{lemma}\label{key-Lemma-for-bounded}
 If $\mu_n \in \cM_{f_n}^e$ is a MME  for $f_n$, then there exist an index $j$ and a subsequence $(\mu_{n_l})_{l\geq 0}$ such that $\mu_{n_l}\sim p_j(f_n)$ for every $l \geq 0$.
    \end{lemma}
\begin{proof} With the results established in Lemmas \ref{Uniformly far away from zero}, \ref{comparation metric entropy}, and \ref{Uniform Pesin Blocks}, the proof follows by employing the same argument as in \cite[Lemma 5.4]{MP24}. 
\end{proof}

Returning to the proof of Theorem \ref{theo-entropy-condition-dimension-2}, for $n$ large enough each $f_n$ has at most $k$ measures of maximal entropy, otherwise, by Lemma \ref{key-Lemma-for-bounded}, there would exit a subsequence $f_{n_l}$ such that every $f_{n_l}$ has at least two different ergodic measures of maximal entropy that are homoclinically related to the same $p_i(f_n)$ which is a contradiction. Thus, the proof is concluded.

\end{proof}

\section{Derived from Anosov}

Recall that $\DA$, with $r>1$, is the set of partially hyperbolic $C^r$ diffeomorphisms of $\mathbb{T}^4$ that admit a dominated splitting $E^{ss}\oplus E^1\oplus E^2 \oplus E^{uu}$, where each subbundle is one-dimensional, and whose linear part $A$ has eigenvalues satisfying $|\lambda^{ss}|<|\lambda^s|<1<|\lambda^u|<|\lambda^{uu}|$   with corresponding eigenspaces $E^{ss}_A, E^{s}_A, E^{u}_A$, and $E^{uu}_A$ respectively.

Let $\pi:\mathbb{R}^4\to \mathbb{T}^4$ be the natural covering map. For $*=s,u,ss,uu$, denote by $\tilde{f}$, $\tilde{\mathcal{F}}^*$ and $\tilde{W}^*$ the lifts of $f$, $\mathcal{F}^*$ and $W^*$, respectively. Here, $\mathcal{F}^{ss}$ and $\mathcal{F}^{uu}$ denote the stable and unstable foliations of $f$, which are tangent to $E^{ss}$ and $E^{uu}$, respectively.

Given $f\in \DA$, we know from \cite{Fra69} that there exists a continuous and surjective map $H: \mathbb{R}^4\to \mathbb{R}^4$ such that  
$A\circ H = H\circ \tilde{f}$,  
and a constant $K>0$ such that  
\begin{equation}\label{close-to-identity}
    d_{\mathbb{R}^4}(H(x),x) < K,
\end{equation}
where $d_{\mathbb{R}^4}$ denotes the Euclidean norm in $\mathbb{R}^4$.

\begin{definition}\label{def-quasi-isometric}
A foliation $\mathcal{F}$ is \emph{quasi-isometric} if there exists a constant $Q > 1$ such that, for any two points $x, y \in M$ lying on the same leaf $\cF(z)$, with $z\in M$, satisfies  
$d_{\mathcal{F}(z)}(x,y) \leq  Q d_M(x,y) + Q$,  
where $d_{\mathcal{F}}$ is the distance along the leaf and $d_M$ is the distance on the manifold.    
\end{definition}

\begin{proposition}\label{quasi-isometric-implies-hu-pequeno}
Let $f\in \DA$ be such that $\tilde{\mathcal{F}}^{uu}$, the lift of unstable foliation, is quasi-isometric. Then,  
$h^u(f) \leq \log |\lambda^{uu}|$.
\end{proposition}

\begin{proof}
   Recall that by \cite{HHW17} the volume growth is equal to the unstable entropy. Hence, as $\cF^{uu}$ is one dimensional, it suffices to show that for any two points $x$ and $y$ on the same unstable manifold $W^{uu}(z)$ with $z\in \mathbb{T}^4$, it holds
    $$
    \limsup_{n\to \infty}\frac{1}{n}\log d_{W^{uu}_z}(f^n x, f^n y) \leq \log |\lambda^{uu}|.
    $$

    Let $x,y\in \mathbb{T}^4$ lie on the same unstable manifold $W^{uu}_z$ of the point $z\in \mathbb{T}^4$,
    and let $\tilde{W}^{uu}_z$, $\tilde{f}$, $\tilde{x}$, and $\tilde{y}$ denote the lifts of $W^{uu}_z$, $f$, $x$, and $y$, respectively. As $W^{uu}_z$ is homeomorphic to $\mathbb{R}$ we have
    $$
    d_{W^{uu}_z}(f^n x, f^n y) = d_{\tilde{W}^{uu}_z}(\tilde{f}^n \tilde{x}, \tilde{f}^n \tilde{y}).
    $$

    Since $\tilde{\cF}^{uu}$ is quasi-isometric, there exists a constant $Q > 0$ such that
    $$
    d_{W^{uu}_z}(f^n x, f^n y) \leq Q d_{\mathbb{R}^4}(\tilde{f}^n \tilde{x}, \tilde{f}^n \tilde{y}) + Q.
    $$

    From \eqref{close-to-identity}, for every $w\in \mathbb{R}^4$, we obtain
    \begin{equation*}
        d_{W^{uu}_z}(f^n x, f^n y) \leq 2 K Q + Q d_{\mathbb{R}^4}(H \tilde{f}^n \tilde{x}, H \tilde{f}^n \tilde{y}) + Q.
    \end{equation*}

    Since $H \circ \tilde{f}^n = A^n \circ H$, we have
    $$
    d_{W^{uu}_z}(f^n x, f^n y) \leq 2 K Q + Q d_{\mathbb{R}^4}(A^n H \tilde{x}, A^n H \tilde{y}) + Q.
    $$

    Taking the upper limit, we conclude that
    $$
    \limsup_{n\to \infty} \frac{1}{n} \log d_{W^u_z}(f^n x, f^n y) \leq \log |\lambda^{uu}|.
    $$

    This implies that $h^u(f) \leq \log |\lambda^{uu}|$, completing the proof.
\end{proof}

Before proving Theorem \ref{theo-DA}, we introduce the following property. We define $\Pi^\sigma_x$ to be the projection of $\mathbb{R}^4$ onto $E^\sigma_A + x$ along the complementary subbundles of $A$, where $\sigma = ss, uu$. Let $H_\sigma := \Pi^\sigma_x \circ H$.

\begin{definition} For $\sigma = ss, uu$, we say that $f \in \DA$ is \emph{$\sigma$-proper} if, for every $R > 0$, there exists $R' > 0$ such that, for every $x \in \mathbb{R}^4$, if $y \in \tilde{W}^\sigma(x)$ and $d(H_\sigma(x), H_\sigma(y)) < R$, then $d_\sigma(x, y) < R'$. \end{definition}

\begin{proposition}{\cite[Proposition 7.1]{FPS14}}\label{proper-implies-quasi-isometric} If $f \in \DA$ is isotopic to $A$ along a path of partially hyperbolic diffeomorphisms and is $\sigma$-proper for $\sigma = ss, uu$, then the foliation $W^\sigma$ is quasi-isometric. \end{proposition}

\begin{proof}[Proof of Theorem \ref{theo-DA}] Let $f \in \DA$ be such that the stable and unstable foliations are quasi-isometric. Denote by $f_A$ the diffeomorphism induced by $A$ on $\mathbb{T}^4$. Since $f$ and $f_A$ are semi-conjugate, we have
$$ h(f) \geq h(f_A) = \log|\lambda^{uu}| + \log|\lambda^{u}| = \log|\lambda^{ss}| + \log|\lambda^{s}|. $$

By Proposition \ref{quasi-isometric-implies-hu-pequeno}, it follows that  

$$ h^u(f) \leq \log|\lambda^{uu}| \quad \text{and} \quad h^s(f) \leq \log|\lambda^{ss}|. $$

Combining these inequalities, we obtain  

$$ h^u(f) < h(f) \quad \text{and} \quad h^s(f) < h(f). $$  

By Theorem \ref{theo-entropy-condition-dimension-2}, this proves the first part of the theorem.

For the second part we proceed as follows.
The proof of \cite[Theorem 6.1]{FPS14} implies that any $f$ that is isotopic to $A$ along a path of partially hyperbolic diffeomorphisms is $\sigma$-proper for $\sigma = ss, uu$. By Proposition \ref{proper-implies-quasi-isometric}, this ensures that the stable and unstable foliations of $f$ are quasi-isometric. 

Therefore, we conclude that if $f$ is isotopic to $A$ along a path of partially hyperbolic diffeomorphisms then $\tilde{\cF}^*$, $*=ss,uu$, is quasi isometric. So, by the previous part, we complete the proof of Theorem \ref{theo-DA}.  
\end{proof}


\section{Proof of Theorem \ref{MainTheo2index} and  \ref{theo-entropy-condition-different index}}
Let $f: M \to M$ be a diffeomorphism on a compact manifold $M$ admitting a dominated splitting 
$$E^{ss} \oplus E_1 \oplus \cdots \oplus E_k \oplus E^{uu},$$ 
where each $E_i$ is one dimensional.

For any $x \in M$ and $i \in {1, \dots, k}$  define:

\begin{equation*}
\begin{split}
E^{s,i}(x) := E^{ss}(x) \oplus E_1(x) \oplus \cdots \oplus E_i(x)\quad
\text{and}
\\\quad E^{u,i}(x) := E_{i+1}(x) \oplus \cdots \oplus E_k(x) \oplus E^{uu}(x).
\end{split}
\end{equation*}

\begin{lemma}\label{Pliss-Lemma}
    Given $\alpha > 0$ and $i \in \{1, \cdots, k\}$, there exist $\ell' \in \mathbb{Z}^+$ and $\beta > 0$ such that 
    $$\mu\left(B^s_\ell(f, E^{s,i})\right) > \beta$$ 
    for all $\alpha$--hyperbolic measures $\mu$ with index $\operatorname{dim}(E^{ss}) + i$ and $\ell>\ell'$.
\end{lemma}

\begin{proof}
    Fix $\rho > 0$ such that $-\alpha < \log \rho < 0$, and define 
    $$A(f) = \left\{x \in M : \|Df^n(x)|_{E^{s,i}}\| \leq \rho^n \text{ for all } n \in \mathbb{Z}^+\right\}.$$

    Exactly as in \cite[Lemma 4.3]{MP24}, using the Pliss Lemma~\cite{Pli72}, there exists $\beta > 0$ such that 
    $$\mu(A(f)) > \beta$$ 
    for every $\alpha$--hyperbolic measure $\mu$ with index $\operatorname{dim}(E^s) + i$.

    Now, choose $\ell' > 0$ such that $\ell' \log(\rho) < -1$. 
    Observe that by the dominated splitting and the fact that $E^i$ is one dimensional we have
    $$
    \|Df^{\ell n}(x)|_{E^{s,i}}\|=\|Df^{\ell n}(x)|_{E^{i}}\|
    =\|\prod_{j=0}^n Df^{\ell}_{f^{j\ell}}(x)|_{E^{i}}\|=\|\prod_{j=0}^n Df^{\ell}_{f^{j\ell}}(x)|_{E^{s,i}}\|.$$

    If $x \in A(f)$ and $\ell>\ell'$, we conclude that 
    $$
    \begin{aligned}
        n\ell \log(\rho) & \geq \log\|Df^{\ell n}(x)|_{E^{s,i}}\| \\
        & = \log\|\prod_{j=0}^n Df^{\ell}_{f^{j\ell}}(x)|_{E^{s,i}}\|,
    \end{aligned}
    $$
    for all $n \in \mathbb{Z}^+$. Hence, $x \in B^s_\ell(f)$.

    Therefore, for any $\alpha$--hyperbolic measure $\mu$ with $\operatorname{index}(\mu) = \dim(E^s) + i$, we get 
    $$\mu(B^s_\ell(f)) > \beta.$$
\end{proof}

\begin{proof}[Proof of Theorem \ref{MainTheo2index}]
Consider $f:M\to M$, $\alpha > 0$ and $i \in \{1, \cdots, k\}$ as in Theorem \ref{MainTheo2index}. Suppose, by contradiction, that $f$ admits infinitely many ergodic measures of maximal entropy with  index equal to $\dim(E^{ss}) + i$. Thus, we take a sequence of ergodic measures of maximal entropy $\{\mu_n\}$, with $\operatorname{index}(\mu_n) = \dim(E^{ss}) + i$ for each $n$, such that $\mu_n$ converges to an invariant measure $\mu$. 

Since each $\mu_n$ is $\alpha$--hyperbolic, by Lemma \ref{Pliss-Lemma}, there exist $\ell_1 \in \mathbb{Z}^+$ and $\beta > 0$ such that for all $\ell > \ell_1$,
$$
\mu_n(B^s_\ell(f, E^{s, i})) > \beta.
$$

On the other hand, by Proposition \ref{upper-semicontinuous, entropy expansive} $\mu$ is a measure of maximal entropy.
  As there are no ergodic measures of maximal entropy with index greater than $\dim(E^{ss})+i$,  the Ergodic Decomposition Theorem implies that

\begin{equation}
    \mu(\Lambda^{cu}_{\alpha}(f,E^{u, i}))=1. 
\end{equation}
 So,  by Lemma \ref{Uniform Pesin Blocks}, there exist $\ell_2 \in \mathbb{Z}^+$ and $N' > 0$ such that for all $\ell > \ell_2$ and $n > N'$,
$$
\mu_n(B^u_\ell(f, E^{u, i})) > 1 - \beta.
$$

Combining these results, for each $\ell > \max(\ell_1, \ell_2)$ and $n > N'$, we have
$$
\mu_n(B_\ell(f, E^{s, i}, E^{u, i})) > 0.
$$

Thus, by Lemma \ref{Uniform-size-implies-homoclinically -related}, there exist $k_1, k_2 \in \mathbb{N}$ such that $\mu_{k_1} \sim \mu_{k_2}$. This contradicts Theorem \ref{BCS-critirium}, completing the proof.
\end{proof}

\begin{proof}[Proof of Theorem \ref{theo-entropy-condition-different index}]
Let $f: M \to M$ and $\alpha > 0$ be as in Theorem \ref{theo-entropy-condition-different index}. By Lemma \ref{Uniformly far away from zero}, we have that  
$$\lambda_1(f, \mu) < \max\{h^u(f), h^s(f)\} - h(f)$$  
and  
$$\lambda_3(f, \mu) > h(f) - \max\{h^u(f), h^s(f)\}$$  
for every MME. Thus,  reducing the value of $\alpha$ if necessary, we can assume that each MME is $\alpha$--hyperbolic and has index $\dim(E^{ss})+1$ or $\dim(E^{ss})+2$. Therefore, to prove that $f$ has finitely many MMEs, it suffices to show that the number of MMEs with index $\dim(E^{ss})+1$ and $\dim(E^{ss})+2$ is finite.

Since $\lambda_3(f,\mu) < 0$ for each MME, there are no MMEs with index equal to $\dim(E^{ss})+3$. By Theorem \ref{MainTheo2index}, the number of MMEs with index $\dim(E^{ss})+2$ is finite. Analogously, applying Theorem \ref{MainTheo2index} to $f^{-1}$, we conclude that the number of MMEs with index $\dim(E^{ss})+1$ is also finite. Thus, $f$ has finitely many MMEs.

Now, to prove the second part of the theorem, consider $\chi>0$ smaller than $\alpha$, otherwise  the proof follows as in Theorem $\ref{theo-entropy-condition-dimension-2}$. Consider $k>0$ being  the number of MMEs  of $f$ denoted by $\nu_1, \ldots, \nu_k$  and take  $p_1(f), \ldots, p_{k}(f)$ being  hyperbolic  periodic points of $f$  given by Proposition \ref{HoclinicClassRelation}. 

Consider a sequence $f_n$ of $C^{1+}$ diffeomorphisms converging to $f$ in the $C^1$ topology.
    Let $p_i(f_n)$ be the unique hyperbolic periodic point of $f_n$ derived from the structural stability of $p_i(f)$.

\begin{lemma}\label{p(g)SimMuj}
 If $\mu_n \in \cM_{f_n}^e$ is an MME for $f_n$, then there exist an index $j$ and a subsequence $(\mu_{n_l})_{l\geq 0}$ such that $\mu_{n_l} \sim p_j(f_n)$ for every $l \geq 0$.
\end{lemma}

\begin{proof} 
Consider a subsequence $(\mu_{n_i})_{i\geq 0}$ of $(\mu_n)_{n\geq 0}$ converging to some $\mu \in \cM_f$, where each $\mu_{n_k}$ has the same index. Without loss of generality, assume that this index is $\dim(E^{ss})+2$. 

By repeating the argument from the proof of Theorem \ref{MainTheo2index}, we can find constants $2\beta > 0$ and $\ell$, both depending on $\chi$, such that 

$$\mu_{n_i}\left(B_\ell(f_n, E^{s,2}_{f_n},E^{u,2}_{f_n})\right) > 2\beta.$$

Let $\tilde{B}_i \subset B_\ell(f_n, E^{s,2}_{f_n},E^{u,2}_{f_n})$ be a compact set of regular points of $\mu_{n_i}$ such that 

$$\mu_{n_i}(\tilde{B}_i) > \beta.$$

Thus, there exists a subsequence $(\mu_{n_l})_{l\geq 0}$ of $(\mu_{n_j})_{j\geq 0}$ and a compact set $$\tilde{B} \subset B_\ell(f, E^{s,2}_{f},E^{u,2}_{f})$$
such that

$$\mu(\tilde{B}) > \beta,$$

and $\tilde{B}_i$ converges to $\tilde{B}$ in the Hausdorff metric.

Since $\mu$ is a measure of maximal entropy, it is a convex combination of finitely many MMEs. Moreover, each MME with an index different from $\dim(E^{ss})+2$ gives zero measure to $\tilde{B}$. Therefore, there exists a MME $\nu_j$ of $f$ with $j \in \{1,\dots, k\}$ such that $\operatorname{index}(\nu_j) = \dim(E^{ss})+2$ and $\nu_j(\tilde{B}) > 0$. Now, like in Lemma~\ref{key-Lemma-for-bounded}, the proof follows as in  \cite[Lemma 5.4]{MP24}.

\end{proof}

The proof is concluded by applying the same argument as in the final paragraph of the proof of Theorem \ref{theo-entropy-condition-dimension-2}.
   
\end{proof}

\section{Proof of ~\ref{theo-entropy-Skew-Product}}
Let $f:M\to M$ and $F:M\times S^1\to M\times S^1$ be as in the hypothesis of ~\ref{theo-entropy-Skew-Product}, let $P:M\times S^1\to M$ be the projection $(x,t)\mapsto x$. 

 For each $x\in M$, we define  $f^n_x=f_{f^{n-1}(x)}\circ \cdots \circ f_{x}$ for $n\in \bZ^+$. Observe that $F^n(x,t)=(f^n(x),f^n_x(t)).$

Using a cone argument, it is easy to prove that $F$ is also partially hyperbolic with a dominated splitting $E^s_F\oplus E^-_F\oplus E^c_F\oplus E^+_F \oplus E^u_F$ with $E^c_F(x,t)=\{0\}\times TS^1_t$, see for example~\cite[Section 3.2]{CP15}.

\begin{lemma}\label{lem.entropy-ineq}
    Let $\tilde{\mu}$ be an $F$-invariant measure such that $P_*\tilde{\mu}=\mu$, then $h^u_{\tilde{\mu}}(F)\leq h^u_{\mu}(f)$. As a consequence, $h^u(F)\leq h^u(f)$.
    
    Moreover, if $\tilde{\mu}$ is a MME of $F$ then $\mu=P_*\tilde{\mu}$ is a MME of $f$ and  $h_{\tilde{\mu}}(F)=h_\mu(f)$.
\end{lemma}

  \begin{proof}

Let $\eta_f$ be an increasing partition subordinated to $\cF^{uu}_f$ . Define $\eta_F$ to be the partition $\eta_F(x,t)=P^{-1}(\eta_f(x))\cap W_F^{uu}(x,t)$, as $P(W^{uu}_F(x,t))=W^{uu}_f(x)$, this is an increasing partition subordinated to $\cF^{uu}_F$ .
  
Fix $z\in M$ and let $R=\eta_f(z)$ and $\widetilde{R}=P^{-1}R$, observe that $f^{-1}\eta_f$ gives a finite partition of $R$, we denote this by $Q_i,i=1,\dots,k$, define $\widetilde{Q}_i:=P^{-1}Q_i$.


Let $\mu_{\widetilde{R}}$ be the measure on $\widetilde{R}$ given by the disintegration of $\widetilde{\mu}$ with respect to the partition $P^{-1}\eta_f$ and $ S^1\ni t \mapsto \mu^{\eta_F}_{t}$ be the disintegration of $\mu_{\widetilde{R}}$ with respect to $\eta_F(z,t)$, define $\mu_R=P_*\mu_{\widetilde{R}}$. 

Let $\hat{\mu}_{\widetilde{R}}$ be the projection along $\cF^{uu}_F$ of $\mu_{\widetilde{R}}$ into $S^1$ .
By Jensen inequality we have
$$
\begin{aligned}
\int_{S^1}-\mu^{\eta_F}_t(\widetilde{Q}_i)\log \mu^{\eta_F}_t(\widetilde{Q}_i)d\hat{\mu}_{\widetilde{R}}&\leq -\mu_{\widetilde{R}}(\widetilde{Q}_i)\log \mu_{\widetilde{R}}(\widetilde{Q}_i)\\
&=-\mu_R(Q_i)\log \mu_R(Q_i).
\end{aligned}
$$

Summing on $i$ we get
$$
        -\int_{\widetilde{R}}\log \mu^{\eta_F}_{x,t}(F^{-1}\eta_F (x,t)) d\mu_{\widetilde{R}}(x,t)\leq -\int_{R}\log \mu_R(f^{-1}\eta_f(x))d\mu_R(x).
        $$

Now, integrating on $z$ we have
$$
h^u_{\widetilde{\mu}}(F)\leq h^u_{\mu}(f).
$$

By \cite{HHW17} we have $h^u(F)\leq h^u(f)$.

For the last part, take $\mu$ an MME of $f$, by \cite{ledrappier-walters} we have that
$$
\sup_{\tilde{\mu};P_*\tilde{\mu}=\mu}h_{\tilde{\mu}}(F)=h_\mu(f)+\int_M h(f,P^{-1}(x))d\mu(x).
$$
Since $P^{-1}(x)=\{x\}\times S^1$, the second term in the above equation is zero, implying that
$ 
\sup_{\tilde{\mu};P_*\tilde{\mu}=\mu}h_{\tilde{\mu}}(F)=h_\mu(f).
$ 
Thus,  $\tilde{\mu}$ is a measure of maximal entropy of $F$ if and only if $P_*\tilde{\mu}$ is a measure of maximal entropy of $f$. Moreover, if $\tilde{\mu}$ is ergodic then $\mu$ is also ergodic.
    \end{proof}
\begin{remark}
The proof of Lemma~\ref{lem.entropy-ineq} follows the same lines of \cite{tahzibi-yang-IP}, in which the base diffeomorphism is hyperbolic. In that result it is proved that the equality occurs if, and only if, the center disintegration of $\tilde{\mu}$ is invariant by the strong holonomies. This can also be proved here without much difficulty.
\end{remark}
Let $x\in M$, recall
$$W^s_{C,\chi,\varepsilon}(f,x)=\{y\in B_\varepsilon(x);d(f^n(x),f^n(y)\leq Ce^{-\chi n}d(x,y),n\geq 0\}.$$

Let $\mu$ be an $f$ invariant $\chi$-hyperbolic measure such that $\lambda^-(f,\mu)<0<\lambda^+(f,\mu)$, where $\lambda^*(f,\mu)$ is the Lyapunov exponent along $E^*$, $*=+,-$.
Recall that for $\mu$ almost every $x$ there exists $C(x),\varepsilon(x)>0$ it is a manifold tangent to $E^{ss}\oplus E^-$,

 For $C,\chi,\varepsilon>0$ let $R_{C,\chi,\varepsilon} \subset M$ be the set of points such that $W^s_{C,\chi,\varepsilon}(f,x)$ is a Pesin Manifold tangent to $E^s\oplus E^-$. The following Lemma gives the existence of holonomies along Pesin manifolds.
\begin{lemma}\label{lemma-holonomy}
    Let $x\in R_{C,\chi,\varepsilon}$, then
    for every $y\in W^s_{C,\chi,\varepsilon}(f,x)$, there exists the $C^0$ limit
    $$
    H^s_{x,y}=\lim_{n\to \infty}(f^n_y)^{-1}\circ f^n_x.
    $$
    This limit is uniform in $C$, $\varepsilon$, $\chi$, and in a $C^1$ neighborhood of $F$. Moreover, these maps are Lipchitz.
\end{lemma}

\begin{proof}
We take $x \in M$ and $y \in W^s_{C,\chi,\varepsilon}(x)$. Define $H^n_{x,y} = (f^n_y)^{-1} \circ f^n_x$, and let $I$ denote the identity map on $S^1$. We use $d_{C^0}$ to denote the distance in the $C^0$--topology. Let $K$ be the Lipschitz constant of the map $z \mapsto f^{-1}_z$.    

Lets estimate $d_{C^0}\left(H^{n+1}_{x,y},H^{n}_{x,y}\right)$.
$$
\begin{aligned}
    d \left(H^{n+1}_{x,y}(t),H^{n}_{x,y}(t)\right)&=   d \left((f^{n}_y)^{-1}\circ f^{-1}_{f^n(y)}\circ f_{f^n(x)}\circ f^n_x(t),(f^n_y)^{-1}\circ f^n_x(t)\right)\\
    &\leq \sup_{a\in S^1}|{f^n_y}'(a)^{-1}|d_{C^0}(f^{-1}_{f^n(y)}\circ f_{f^n(x)},I)\\
    &\leq \sup_{a\in S^1}|{f^n_y}'(a)^{-1}|d_{C^0}(f^{-1}_{f^n(y)},f^{-1}_{f^n(x)})\\
    &\leq K\sup_{a\in S^1}|{f^n_y}'(a)^{-1}|  \|Df^n\mid_{E^-(\theta)}\| d(x,y)\\
    &\leq K\frac{\sup_{a\in S^1}|{f^n_y}'(a)^{-1}|}{\sup_{a\in S^1}|{f^n_\theta}'(a)^{-1}|}  \sup_{a\in S^1}|{f^n_\theta}'(a)^{-1}|\|Df^n\mid_{E^-(\theta)}\| d(x,y),
\end{aligned}
$$
where $\theta\in W^s_{C,\chi,\varepsilon}(f,x)$. The map $z\mapsto \sup_{a\in S^1}|{f^n_\theta}'(z)^{-1}|$ is H\"older continuous, so by a classical distortion argument we have that there exists $L=L(C,\chi,\varepsilon)$ such that $$\frac{\sup_{a\in S^1}|{f^n_y}'(a)^{-1}|}{\sup_{a\in S^1}|{f^n_\theta}'(a)^{-1}|}\leq L.$$

By \eqref{eq.skewP} and taking the supremum over $t\in S^1$ we get
$$
 d_{C^0} \left(H^{n+1}_{x,y},H^{n}_{x,y}\right)\leq K L \gamma^n d(x,y).
$$
Then $(f^n_y)^{-1}\circ f^n_x$ is a Cauchy sequence, uniform in $C,\chi,\varepsilon$, hence it converges uniformly.

By triangular inequality we also have $d_{C^0}\left(H^n_{x,y},Id\right)\leq KL\sum_{j=0}^n \gamma^j d(x,y)$ so we conclude that there exists $\tilde{C}>0$, such that $d_{C^0}(\left(H^s_{x,y},Id\right)\leq \tilde{C}d(x,y)$.

Now let us prove that $H^s_{x,y}$ is Lipchitz, fix $t,s\in S^1$. By the Mean Value Theorem
\begin{equation}\label{eq.valor.medio}
 d(H^n_{x,y}(t),H^n_{x,y}(s))\leq |{H^n_{x,y}}'(\theta)|d(t,s)   
\end{equation}
for some $\theta\in S^1$. Lets find a uniform bound for  $|{H^n_{x,y}}'(\theta)|$.

$$
\begin{aligned}
{H^n_{x,y}}'(\theta)&={f^n_y}'(H^n_{x,y}(\theta))^{-1} {f^n_{x}}'(\theta)\\
&=\prod_{j=0}^{n-1}f'_{f^j(y)}(f^j_y H^n_{x,y}(\theta))^{-1} \prod_{j=0}^{n-1}f'_{f^j(x)}(f^j_x (\theta))\\
&=\prod_{j=0}^{n-1}\frac{f'_{f^j(y)}(f^j_x (\theta))}{f'_{f^j(y)}(f^j_y H^n_{x,y}(\theta))}\prod_{j=0}^{n-1}\frac{f'_{f^j(x)}(f^j_x (\theta))}{f'_{f^j(y)}(f^j_x (\theta))}.
\end{aligned}
$$
Let $\alpha=\min\{r-1,1\}$. Then, the map $x,\theta\mapsto f'_x(\theta)$ is $\alpha$-H\"older

For the first term observe that $f^j_y \circ H^n_{x,y}= H^{n-j}_{f^j(x),f^j(y)}\circ f^j_x$, then 
$$
\begin{aligned}
|\log\left|\prod_{j=0}^{n-1}\frac{f'_{f^j(y)}(f^j_x (\theta))}{f'_{f^j(y)}(H^{n-j}_{f^j(x),f^j(y)} f^j_x(\theta))}\right||&\leq C_2 \sum_{j=1}^n d(H^{n-j}_{f^j(x),f^j(y)},Id)^\alpha\\
&\leq C_2 \tilde{C} \sum_{j=0}^n d(f^j(x),f^j(y))^\alpha \\
&\leq  C_2 \tilde{C} \sum_{j=0}^n e^{-\chi\alpha j} \varepsilon^\alpha
\end{aligned}
$$
for some $C_2>0$.

For the second term
$$
|\log \left|\prod_{j=0}^{n-1}\frac{f'_{f^j(x)}(f^j_x (\theta))}{f'_{f^j(y)}(f^j_x (\theta))}\right||\leq C_1 C \sum_{j=0}^{n-1} e^{-\chi \alpha j}d(x,y)
$$
for some $C_1>0$.

 Thus, we get that there exists some $Q>0$, depending only on $C,\chi,\varepsilon$, such that $|{H^n_{x,y}}'(\theta)|\leq Q$.
 Then, making $n\to \infty$ on \eqref{eq.valor.medio} we get
$$
d(H^s_{x,y}(t),H^s_{x,y}(s))\leq Qd(t,s).
$$
\end{proof}

As a direct consequence on the uniformity of the convergence we get the following. 
\begin{corollary}\label{cor.cont-hol}
    Fix $C,\chi,\varepsilon$, then the function $x,y,F\mapsto H^s_{x,y}$ is continuous in the $C^0$ topology in the domain 
    $$\{(x,y,F):x\in R_{C,\chi,\varepsilon},y\in W^s_{C,\chi,\varepsilon}(f,x)\text{ and }F\in \SP\}.$$
\end{corollary}

We will need the following Theorem from \cite{BCS22}
\begin{theorem}[{\cite[Theorem~3.1]{BCS22}}]\label{BCS} Let $f:M\to M$ be a diffeomorphims and $\mu$ be a $\chi$-hyperbolic measure, then there exists a topological Markov shift $(\Sigma, \sigma)$ and a H\"older continuous map $\pi\colon \Sigma\to M$ such that $\pi\circ \sigma=f\circ \pi$. Moreover:
\begin{enumerate}
    \item $\Sigma$ is irreducible.
    \item  $\nu(\pi(\Sigma))=1$ for every $\chi$ hyperbolic measure $\nu$ homoclinically related to $\mu$.
    \item $\pi$ is finite to one in a total measure set, for every $\sigma$ invariant measure.
    \item $h_\nu(f)=h_{\hat{\nu}}(\sigma)$ for any $\hat{\nu}$ $\sigma$-invariant measure such that $\nu=\pi_*\hat{\nu}$.
\end{enumerate}
    \end{theorem}

    The fact of $\pi$ being finite to one, allows to lift any $f$-invariant measure $\nu\sim \mu$ to a $\sigma$ invariant measure $\hat{\nu}$ such that $\pi_* \hat{\nu}=\nu$. See for example \cite{Sar13}.


Fix $\chi = h(f) - \max\{h^u(f), h^s(f)\}$. Recall that if $\mu$ is an MME of $f$, then by Lemma \ref{h(f)>h^s(f),h^u(f) implies far away from zero}, $\mu$ is $\chi$--hyperbolic. From now on we fix $\mu$ a MME of $f$ and $\sigma:\Sigma\to \Sigma$ given by Theorem~\ref{BCS}.



Now we lift $F$ as a skew product over $\sigma$ and define
$$\hat{F}:\Sigma\times S^1\to \Sigma\times S^1,\quad (\hat{x},t)\mapsto (\sigma(\hat{x}),f_{\pi(\hat{x})}).$$ 

Observe that the map $\tilde{\pi}:\Sigma\times S^1\to M \times S^1$, $(\hat{x},t)\mapsto (\pi(\hat{x}),t)$ satisfies $F\circ \tilde{\pi}=\tilde{\pi}\circ \hat{F}$.

Given $\widetilde{\mu}$ an $F$-invariant measure such that $P_*\widetilde{\mu}= \mu$, let $x\mapsto \mu^c_x$ be its center disintegration. Then we can lift $\widetilde{\mu}$ to $\Sigma\times S^1$ as the measure $m=\int_{\Sigma} \mu^c_{\pi(\hat{x})}d\hat{\mu}(\hat{x})$ where $\hat{\mu}$ is such that $\pi_*\hat{\mu}=\mu$,  $m$ is $\hat{F}$ invariant and $\tilde{\pi}_* m=\widetilde{\mu}$. 

Resuming, we have the following commuting diagrams:
\begin{itemize}
    \item The skew product structure of $F$ gives
$$
\begin{array}{cccccc}
F:& M\times S^1     & \to & M\times S^1 ,& \,& \widetilde{\mu}\\
     &  P \downarrow & & \downarrow& P_*& \downarrow  \\
f:& M     & \to & M ,& \,&\mu.
\end{array}
$$
\item From theorem~\ref{BCS} we have
$$
\begin{array}{ccccc}
\sigma:&    \Sigma & \to &\Sigma, &\hat{\mu}\\
 &  \pi \downarrow & & \downarrow&  \downarrow \pi_* \\
f:& M     & \to & M ,&\mu.
\end{array}
$$

\item The lift of $F$ to the $\Sigma\times S^1$ gives

\begin{equation}\label{eq.liftF}
\begin{array}{ccccc}
\hat{F}:&    \Sigma\times S^1 & \to &\Sigma\times S^1, &m\\
 &  \tilde{\pi} \downarrow & & \downarrow&  \downarrow {\tilde{\pi}_*} \\
F:& M\times S^1     & \to & M\times S^1 ,&\widetilde{\mu}.
\end{array}
\end{equation}
\end{itemize}


Denote $\hat{x}=(x_i)_{i\in \mathbb{Z}}\in \Sigma$, as usual define 
$$W^s(\sigma,\hat{x})=\{\hat{y}\in \Sigma;x_i=y_i,\, i\geq 0\} \text{ and } W^u(\sigma,\hat{x})=\{\hat{y}\in \Sigma;x_i=y_i,\, i\leq 0\}. $$

\begin{definition}
For $\hat{y}\in W^s(\sigma,\hat{x})$ we define the \emph{stable holonomy} $H^s_{\hat{x},\hat{y}}=H^s_{\pi(\hat{x}),\pi(\hat{y})}$. Analogously we define unstable holonomies for $\hat{y}\in W^u(\sigma,\hat{x})$
\end{definition}

\begin{lemma}\label{lema-holonomies-sigma}
    The stable holomies are well defined, they are Lipchitz and satisfy 
    \begin{itemize}
        \item $H^s_{\hat{x},\hat{y}}=H^s_{\hat{z},\hat{y}}\circ H^s_{\hat{x},\hat{z}}$, $H^s_{\hat{x},\hat{x}}=Id$, for every $\hat{z},\hat{y}\in W^s(\sigma,\hat{x})$,
        \item $f_{\pi(\hat{y})}\circ H^s_{\hat{x},\hat{y}}=H^s_{\sigma(\hat{x}),\sigma(\hat{y})}\circ f_{\pi(\hat{x})}$,
        \item The map $\hat{x},\hat{y}\mapsto H^s_{\hat{x},\hat{y}}$ is continuous on pairs $\hat{y}\in W^s(\sigma,\hat{x}).$
    \end{itemize}
\end{lemma}
\begin{proof}
The holonomy is well defined because $\pi(W^s(\sigma,\hat{x}))$ is a Pesin manifold $W^s(f,\pi(\hat{x}))$, so by Lemma~\ref{lemma-holonomy} it exists. 

By \cite[Proposition~4.4]{Ova18}, the constant $C,\chi,\varepsilon$ are uniform in cylinders $[a]:=\{\hat{x}\in \Sigma;x_0=a\}$, so by corollary~\ref{cor.cont-hol} the function $\hat{x},\hat{y}\mapsto H^s_{\hat{x},\hat{y}}$ is continuous.

The other properties follows directly from the definition.
\end{proof}

An $\hat{F}$ invariant measure $m$ has $s$-\emph{invariant center disintegration} if $(H^s_{\hat{x},\hat{y}})_*\mu^c_{\hat{x}}=\mu^c_{\hat{y}}$ for $\hat{\mu}$ almost every $\hat{x},\hat{y}$ with $\hat{y}\in W^s(\sigma,\hat{x})$. Analogously, $m$ is  $u$-\emph{invariant} replacing $s$ by $u$. Moreover, $m$ is  a \emph{continuous $su$-invariant disintegration} if the map $\hat{x}\mapsto \mu^c_{\hat{x}}$ has a continuous extension to $\Sigma$ that is both $s$ and $u$ invariant.

\begin{proposition}\label{invariance-principle}
    Let $\mu$ be a MME of $f$ and $\tilde{\mu}_n$ be $F$-invariant measures such that $P_*\tilde{\mu}_n=\mu$ and 
    $|\lambda^c(\tilde{\mu}_n)|\to 0 $. Then $\hat{F}$ admits an invariant measure that has $su$-invariant continuous center disintegration and projects to $\hat{\mu}$.
\end{proposition}
\begin{proof}
Observe that $F:\Sigma\times S^1\to \Sigma\times S^1$ is a smooth cocycle over the hyperbolic homeomorphism $\sigma:\Sigma \to \Sigma$ and by Lemma~\ref{lema-holonomies-sigma} it admits holonomies, in the sense of \cite{AV-IP}. 
Moreover, $\mu$ is a measure of maximal entropy of $f$ implying that the lift $\hat{\mu}$ on $\Sigma$ is a measure of maximal entropy of $\sigma$. Since $\Sigma$ is irreducible, then $\hat{\mu}$ has full support and product structure, see~\cite{buzzi-sarig-03}.

Take $m_n$ to be lifts of $\tilde{\mu}_n$ on $\Sigma\times S^1$, as $m_n$ projects on $\Sigma$ to a fix invariant measure $\hat{\mu}$ and $S^1$ is compact, the family $m_n$ is tight then, up to taking a subsequence, we can assume $m_n$ converges to $m$. Now by \cite[Theorem~D]{AV-IP} $m$ admits a continuous $su$-invariant disintegration on $supp(\hat{\mu})=\Sigma$.
    \end{proof}

Let $F\in \SP$, let $\mu_1,\dots,\mu_k$ be the MME of $f$. We say that $F$ is   
\begin{itemize}
    \item \emph{Pinching} if for $j=1,\dots,k$ there exist $\chi$-hyperbolic $\ell_j$-periodic points $p_j\sim \mu_j$ such that $f_{p_j}^{\ell_j}:S^1\to S^1$ is Morse-Smale.
    \item \emph{Twisting} if for each $p_j$ there exists $z_j\in W^s(p_j)\pitchfork W^u(p_j)$ such that $H^s_{z_j,p_j}\circ H^u_{p_j,z_j}$ do not preserve the periodic points of $f_{p_j}^{\ell_j}$.
\end{itemize}

\begin{lemma}\label{lem.PandT}
    If $F\in \SP$ is pinching and twisting then there exists $\beta>0$ such that every $\tilde{\mu}$ MME of $F$ has $\lambda^c(F,\tilde{\mu})\geq \beta$.
\end{lemma}
\begin{proof}
Assume, by contradiction, that there exist MMEs $\tilde{\mu}_n$ such that $|\lambda^c(F,\tilde{\mu}_n)|\to 0$.

By Lemma~\ref{lem.entropy-ineq} we have $h_{P_*\tilde{\mu}_n}(f)=h_{\tilde{\mu}_n}(F)$.
Hence $P_*\tilde{\mu}_n$ is a MME of $f$.
Since  Theorem~\ref{theo-entropy-condition-dimension-2} ensures that the number of MMEs is finite, up to taking a subsequence we can assume that there is $\mu$ such that  $P_*\tilde{\mu}_n=\mu$ for every $n\geq 0$.

Let $\sigma:\Sigma\to \Sigma$ be given by Theorem~\ref{BCS} for $f$ and $\mu$ and let $\hat{F}:\Sigma\times S^1\to \Sigma\times S^1$ be the lift of $F$ as in \eqref{eq.liftF}. 

Let $p$ and $z$ be the points given by the pinching and twisting condition homoclinically related to $\mu$.
By Theorem~\ref{BCS}, there exist $\hat{p},\hat{z}\in \Sigma$ such that $\pi(\hat{p})=p$ and $\pi(\hat{z})=z$. 

By Proposition~\ref{invariance-principle} $\hat{F}$ has an invariant measure $m$ that has $su$-invariant continuous disintegration $\hat{x}\mapsto m_{\hat{x}}$.
By invariance of $m$, ${f^{\ell}_p}_*m_{\hat{p}}=m_{\hat{p}}$. Then $m_{\hat{p}}$ is atomic and the atoms are contained in the set of periodic points of $f^{\ell}_p$.

Now, by the $su$-invariance of the disintegration $H^s_{\hat{z},\hat{p}}\circ {H^u_{\hat{p},\hat{z}}}_* m_{\hat{p}}=m_{\hat{p}}$, this is a contradiction because the atoms are not preserved by $H^s_{\hat{z},\hat{p}}\circ {H^u_{\hat{p},\hat{z}}}$.

So we conclude that there exists $\beta>0$ such that $|\lambda^c(F,\tilde{\mu})|\geq \beta$, for every $\tilde{\mu}$ MME, finishing the proof.
\end{proof}
\begin{proof}[Proof of theorem~\ref{theo-entropy-Skew-Product}]
    
    By Lemma~\ref{lem.entropy-ineq} we have that $h(F)=h(f)>h^u(f)\geq h^u(F)$ so, by Lemma~\ref{lem.PandT} and Theorem~\ref{theo-entropy-condition-different index} we only need to prove that pinching and twisting is $C^r$ dense and $C^1$ open in $\SP$. 
    
  First, we prove that pinching and twisting maps are $C^r$-dense.
    
    Let $F\in \SP$.  Theorem~\ref{theo-entropy-condition-dimension-2} ensures that $f:M\to M$ has a finite number of MMEs. 
Let $\mu$ be one of them.

    By Proposition~\ref{HoclinicClassRelation}, there exists   a  $\chi$-hyperbolic $f$-periodic point $p\sim \mu$ and $z\in W^u(p)\pitchfork W^s(p)$.
    Since  Morse Smale diffeomorphisms are $C^r$-dense in $S^1$, we can perturb $F$, only on the second coordinate, so that $f^{\ell}_p:S^1\to S^1$ becomes Morse-Smale, where $\ell$ is the period of $p$.
    
    Next, observe that we can express  $H^s$ as
    $$H^s_{z,p}=\lim_{n\to \infty}(f^n_p)^{-1}\circ f^{n-1}_{f(z)}\circ f_z.$$
     Take a small neighborhood $B_\delta(z)$ that do not contain any other point of the orbits of $z$ and $p$.
      Then, we can perform a $C^r$ small perturbation, affecting only the second coordinate, to obtain a modified  $\tilde{F}$ such  that  
      $\tilde{F}=F$ outside of $B_\delta(z)$ and 
      $$\tilde{f}_z=f_z\circ R_\theta,$$ for a small  $\theta>0$, where $R_\theta$ is the rotation by angle $\theta$.
    
    As $ H^u_{p,z}$ only depends of the values of $f_{f^n(z)}$ and $f_{f^n(z)}$ for $n\leq -1$, we have that for $\tilde{F}$, $\tilde{H}^s_{z,p}\circ \tilde{H}^u_{p,z}=H^s_{z,p}\circ R_\theta\circ  H^u_{p,z} $. 
    Thus, we can find $\theta$ such that $\tilde{H}^s_{z,p}\circ \tilde{H}^u_{p,z}$ do not preserve the periodic points of $f^\ell_p$.

 As we can do this for each of the MMEs of $f$ with $C^r$ small perturbations, we conclude that the pinching and twisting conditions are dense.

To see that they are $C^1$ open, assume that $F$ is pinching and twisting and let $p_i,z_i$, $ i=1,\dots,k,$ be given by the pinching and twisting condition.

Take $G\in \SP$, $C^1$-close to $F$ such that for every $i=1,\dots,k,$ the hyperbolic continuation $p_i(g)$ and $z_i(g)$ of $p_i$ and $z_i$,   
satisfy that $g^{\ell_i}_{p_(g)}$ is $C^1$ close to $f^{\ell_i}_{p_i}$ such that $g^{\ell_i}_{p(g)}$ is Morse Smale and the periodic points are close to the ones of $f^{\ell_i}_{p_i}$. Also, by Corollary \ref{cor.cont-hol}, $H^{s,G}_{z_i(g),p_i(g)}\circ H^{u,G}_{p_i(g),z_i(g)}$ is close to $H^{s,F}_{z,p}\circ H^{u,F}_{p,z}$, so we if $G$ is $C^1$-close $H^{s,G}_{z_i(g),p_i(g)}\circ H^{u,G}_{p_i(g),z_i(g)}$ also do not preserve the periodic points of $g^{\ell_i}_{p_i(g)}$.

Now observe that that by Lemma~\ref{p(g)SimMuj} for every  measure of maximal entropy $\mu_g$ of $g$ there exist $p(g)$ and $z(g)$, continuation  of $p_i$ and $z_i$, so that $p(g)\sim \mu_g$. So $G$ is pinching and twisting.
\end{proof}

{\em{Acknowledgements.}} We are thankful to Mateo Ghezal, Yuri Lima, Davi Obata, Rafael Potrie and Jiagang Yang for helpful conversations about this work.

\bibliographystyle{alpha}
\bibliography{bib}

\end{document}